\documentclass[final,onefignum,onetabnum]{siamart171218}

\usepackage{a4wide}
\usepackage{amsfonts}
\usepackage{amscd}
\usepackage{amssymb}
\usepackage{amsmath}
\usepackage{latexsym}
\usepackage{dsfont}
\usepackage{subfigure}
\usepackage{placeins}

\usepackage{floatflt}
\usepackage{nicefrac}
\usepackage{showkeys}
\usepackage{MnSymbol}
\usepackage[normalem]{ulem}
\usepackage{pgfplots}
\pgfplotsset{compat=newest}
\usepackage{booktabs}
\usepackage{epstopdf}

\usepackage{mathtools}
\DeclarePairedDelimiter{\ceil}{\lceil}{\rceil} %Ceiling operator
\usepackage{esvect} %Vector arrows

\theoremstyle{plain}
\newtheorem{thm}{Theorem}[section]
\theoremstyle{plain}
\newtheorem{lem}[thm]{Lemma}

\newtheorem{cor}[thm]{Corollary}
\newtheorem{defi}[thm]{Definition}
\newtheorem{rem}[thm]{Remark}
\newtheorem{assumption}[thm]{Assumption}

\newtheorem{ex}[thm]{Example}

{%
	\setcounter{enumi}{0}
	\renewcommand{\theenumi}{(\textbf{A}.\arabic{enumi})}
	
	\begin{enumerate}}%
	{\end{enumerate} }
{%
	\setcounter{enumi}{0}
	\renewcommand{\theenumi}{(\textbf{D}.\arabic{enumi})}
	
	\begin{enumerate}}%
	{\end{enumerate} }

\newenvironment{enum_i}
{\begin{enumerate}}
	{\end{enumerate}}

\newcommand{\D}{\ensuremath{\mathcal{D}}}

\newcommand{\eps}{\ensuremath{\varepsilon}}
\newcommand{\R}{\ensuremath{\mathbb{R}}}
\newcommand{\N}{\ensuremath{\mathbb{N}}}

\newcommand{\E}{\ensuremath{\mathbb{E}}}

\newcommand{\ga}{\alpha}
\newcommand{\gb}{\beta}
\newcommand{\gd}{\delta}

\renewcommand{\gg}{\gamma}
\newcommand{\gk}{\kappa}
\newcommand{\gl}{\lambda}
\newcommand{\go}{\omega}
\newcommand{\gs}{\sigma}

\newcommand{\gG}{\Gamma}
\newcommand{\gL}{\Lambda}
\newcommand{\gO}{\Omega}

\newcommand{\cA}{\mathcal{A}}
\newcommand{\cB}{\mathcal{B}}

\newcommand{\cD}{\mathcal{D}}

\newcommand{\cK}{\mathcal{K}}

\newcommand{\cN}{\mathcal{N}}
\newcommand{\cO}{\mathcal{O}}
\newcommand{\cP}{\mathcal{P}}

\newcommand{\cT}{\mathcal{T}}
\newcommand{\cU}{\mathcal{U}}
\newcommand{\cV}{\mathcal{V}}

\newcommand{\1}{\mathbf{1}}
\newcommand{\bP}{\mathbb{P}}

\newcommand{\var}{{\rm Var}}

\newcommand*{\lrscript}[5]{{\vphantom{#1}}_{#2}^{#3}{#1}_{#4}^{#5}}
\newcommand*{\dualpair}[4]{\ensuremath{\lrscript{\langle}{#1}{}{}{} #3 ,
		#4 \rangle_{#2}}}

\makeatletter
\usepackage{color}

\makeatother

\newcommand{\be}{\begin {equation}}
\newcommand{\ee}{\end  {equation}}
\numberwithin{equation}{section} \allowdisplaybreaks[1]

\newcommand{\KL}{{Karhunen-Lo\`{e}ve }}

\definecolor{darkgreen}{rgb}{0,.6,0}

\newcommand{\MC}{Monte Carlo}
\newcommand{\MLMC}{multilevel \MC}

\newcommand{\grad}{\nabla}
\newcommand{\ol}{\overline}

\newcommand{\bee}{\begin {equation*}}
\newcommand{\eee}{\end {equation*}}

\title{A study of elliptic partial differential equations with jump diffusion coefficients}

\author{Andrea Barth \thanks{IANS\textbackslash SimTech, University of Stuttgart 
(\email{andrea.barth@mathematik.uni-stuttgart.de}).}
\and Andreas Stein \thanks{IANS\textbackslash SimTech, University of Stuttgart 
	(\email{andreas.stein@mathematik.uni-stuttgart.de}).}}

\headers{A study of elliptic PDEs with jump diffusion coefficients ]}{A. Barth and A. Stein}

\begin{document}
	\maketitle
	
	\begin{abstract}
		As a simplified model for subsurface flows elliptic equations may be utilized. Insufficient measurements or uncertainty in those are commonly modeled by a random coefficient, which then accounts for the uncertain permeability of a given medium.  
		As an extension of this methodology to flows in heterogeneous\textbackslash fractured\textbackslash porous media, we incorporate jumps in the diffusion coefficient. These discontinuities then represent transitions in the media.
		More precisely, we consider a second order elliptic problem where the random coefficient is given by the sum of a (continuous) Gaussian random field and a (discontinuous) jump part. 
		To estimate moments of the solution to the resulting random partial differential equation, we use a pathwise numerical approximation combined with multilevel Monte Carlo sampling. 
		In order to account for the discontinuities and improve the convergence of the pathwise approximation, the spatial domain is decomposed with respect to the jump positions in each sample, leading to path-dependent grids. 
		Hence, it is not possible to create a sequence of grids which is suitable for each sample path a-priori. 
		We address this issue by an adaptive multilevel algorithm, where the discretization on each level is sample-dependent and fulfills given refinement conditions.  
	\end{abstract}
	

% REQUIRED

\section{Introduction}\label{sec:Intro}
Uncertainty quantification plays an increasingly important role in a wide range of problems in the Engineering Sciences and Physics. Examples of sources of uncertainty are imprecise or insufficient measurements and noisy data. In the underlying dynamical system this is modeled via a stochastic operator, stochastic boundary conditions and/or stochastic data. As an example, to model subsurface flow more realistically the coefficient of an (essentially) elliptic equation is assumed to be stochastic. A common approach in the literature is to use (spatially) correlated random fields that are built from uniform distributions or colored log-normal fields. The resulting marginal distributions of the field are (shifted) normally, resp. log-normally distributed. Neither choice is universal enough to accommodate all possible types of permeability, especially not if fractures are incorporated (see~\cite{ZK04}), the medium is very heterogeneous or porous. 

The last decade has been an active research period on elliptic equations with random data. A non-exhaustive list of publications in this field includes~\cite{ABS13, BNT07, BTZ04, BSZ11, CGST11, CDS11, FST05, LWZ16, NTW08a, SG11, G13}. One can find various ways to approximate the distribution or moments of the solution to the elliptic equation. Next to classical Monte Carlo methods, their multilevel variants and other variance reduction techniques have been applied. The concept of \MLMC\, simulation has been developed in~\cite{H01} to calculate parametric integrals and has been rediscovered in~\cite{G08} to estimate the expected value of functionals of stochastic differential equations. 
Ever since, \MLMC\,techniques have been successfully applied to various problems, for instance in the context of elliptic random PDEs in \cite{ABS13, BSZ11, CGST11, LWZ16, G13} to just name a few. These sampling algorithms are fundamentally different from approaches using Polynomial Chaos. The latter suffer from the fact that one is rather restricted when it comes to possibilities to model the stochastic coefficient. While in the case of fields built from uniform distributions or colored log-normal fields these algorithms can outperform sampling strategies, approaches -- like stochastic Galerkin methods -- are less promising in our discontinuous setting due to the rather involved structure of the coefficient. In fact, it is even an open problem to define them in the case that a L\'evy-field is used. 

Our main objective in this paper is to show existence and uniqueness of the solution to the elliptic equation when the coefficient is modeled as a jump-diffusion. By that we mean a field which consists of a deterministic, a Gaussian and non-continuous part. As we show in the numerical examples, this jump-diffusion coefficient can be used to model a wide array of scenarios. This generalizes the work in~\cite{LWZ16} and uses partly~\cite{HSST12}. To approximate the expectation of the solution we develop and test, further, variants of the multilevel Monte Carlo method which are tailored to our problem: namely adaptive and bootstrapping~\MLMC\ methods. Adaptivity is needed in the jump-diffusion setting, since the coefficient is not continuous. Our analysis shows that the non-adaptivity in a multilevel setting with non-continuous coefficients entails a larger error than when an adaptive algorithm is used. This result is not surprising, since the essence of the multilevel algorithm is that many samples are calculated on coarse grids, where the distributional error from misjudging jump-locations is high. Adaptivity, however, comes to the price that in certain scenarios solving the underlying system of equation becomes computationally expensive. In these settings the advantageous time-to-error performance of adaptive methods may be worse. Bootstrapping, as a simplified version of \textit{Multifidelity Monte Carlo} sampling (see~\cite{GPW16}), reuses samples across levels and is preferred when sampling from a certain distribution is computationally expensive. The bootstrapping algorithms outperform, in general, algorithms with a standard sampling strategy of multilevel Monte Carlo, as it actually reduces the mean square error.

In Section~\ref{sec:setting} we introduce the model problem, define  \textit{pathwise weak solutions} of random partial differential equations (PDEs) and show almost sure existence and uniqueness under relatively weak assumptions on the model parameters. 
The main contribution of this section is the existence and uniqueness result in Theorem~\ref{thm:pw_u}, which is then readily transferred to the special case of a jump-diffusion problem in the subsequent sections of this article. In Section~\ref{sec:a_conv} we define the jump-diffusion coefficient and construct suitable approximations. Both stochastic parts of the jump-diffusion coefficient are approximated: The Gaussian one by a standard truncation of its basis representation, and the jump part (if direct sampling from the jump distribution is not possible) by a technique based on Fourier Inversion. We show $L^p$-type convergence for all existing moments of the approximation. From this result convergence of the approximated solution follows immediately. The approximated solution has still to be discretized to actually estimate moments of it (Section~\ref{sec:fem}). The Galerkin-type discretization is directly furthered into an adaptive scheme. Section~\ref{sec:mlmc} then introduces the sampling methods that are used, namely Monte Carlo, multilevel Monte Carlo and a bootstrapping variant of it. An extensive discussion of numerical examples in one and two dimensions, in Section~\ref{sec:num}, concludes the paper.

%%%%%%%%%%%%%%%%%%%%%%%%%%%%%%%%%%%%%%%%%%%%%%%%%
%%%%%%%%%%%%%%%%%%%%%%%%%%%%%%%%%%%%%%%%%%%%%%%%%
\section{Elliptic boundary value problems and existence of solutions}\label{sec:setting}
%%%%%%%%%%%%%%%%%%%%%%%%%%%%%%%%%%%%%%%%%%%%%%%%%
We consider the following random elliptic equation in a general setting before we specify in Section~\ref{sec:a_conv} our precise choice of coefficient function.
Let $(\gO,\cA,\bP)$ be a complete probability space and $\cD\subset\R^d$, for some  $d\in\N$, be a bounded and connected Lipschitz domain. 
In this paper we consider the linear, random elliptic problem
\begin{equation}\label{eq:elliptic}
\begin{split}
-\grad\cdot \left(a(\go,x)\grad u(\go,x)\right)&=f(\go,x)\quad\text{in $\gO\times\cD$},\\
\end{split}
\end{equation}
where
$a:\gO\times\cD\to\R$ is a stochastic jump-diffusion coefficient and
$f:\gO\times\cD\to\R$ is a random source function. 
The Lipschitz boundary $\partial \cD$ consists of open $(d-1)$-dimensional manifolds which are grouped into two disjoint subsets $\gG_1$ and $\gG_2$ such that 
$\gG_1\neq\emptyset$ and $\partial \cD=\gG_1\cup\gG_2$.
We impose mixed Dirichlet-Neumann boundary conditions 
\begin{align}\label{eq:boundary}
\begin{split}
u(\go,x)&=0\quad\text{on $\gO\times\gG_1$},\\
a(\go,x)\vv n\cdot\grad u(\go,x)&=g(\go,x)\quad\text{on $\gO\times\gG_2$},
\end{split}
\end{align}
on Eq.~\eqref{eq:elliptic}, where $\vv n$ is the outward unit normal vector to $\gG_2$ and $g:\gO\times\gG_2\to\R$,
assuming that the exterior normal derivative $\vv n\cdot\grad u$ on $\gG_2$ is well-defined for any $u\in C^1(\ol\cD)$.
To obtain a pathwise variational formulation of this problem, we use the standard Sobolev space $H^1(\cD)$ equipped with the norm 
\begin{equation*}
||v||_{H^1(\cD)}=\left(\int_\cD|v|^2+||\grad v||_2^2dx\right)^{1/2}\quad\text{for $v\in H^1(\cD)$},
\end{equation*}
where $||\cdot||_2$ denotes the Euclidean norm on $\R^d$.
On the Lipschitz domain $\D$, the existence of a bounded, linear operator  
$$T:H^1(\cD)\to H^{1/2}(\partial\cD)$$
with 
\begin{equation*}
T:H^1(\cD)\cap C^\infty(\overline\cD)\to H^{1/2}(\partial\cD),\quad v\mapsto v\vert_{\partial\cD}
\end{equation*}
and
\begin{equation}\label{eq:tracebound}
||Tv||_{H^{1/2}(\partial\cD)}\le C_\cD ||v||_{H^1(\D)} 
\end{equation}
for $v\in H^1(\cD)$ and some constant $C_\cD>0$, dependent only on $\D$, is ensured by the trace theorem, see for example~\cite{D96}. 
At this point, one might argue that the trace operator $T$ needs to be defined pathwise for any $\go\in\gO$, 
since the Neumann part of the boundary conditions in Eq.~\eqref{eq:boundary} may contain a random function $g:\gO\times\gG_2\to\R$. 
This is true if one works with $T$ on $\gG_2\subset\partial\D$, as the trace $Tv$ then has to match the boundary condition given by $g(\go,\cdot$) on $\gG_2$ for $\bP$-almost all $\go\in\gO$ and $v\in H^1(\D)$. 
In our case, for simplicity, we may treat $T$ independently of $\go\in\gO$, since we only consider the trace operator on the homogeneous boundary part $\gG_1$ to define $V$ as follows:
The subspace of $H^1(\cD)$ with zero trace on $\gG_1$ is then 
\begin{equation*}
V:=\{v\in H^1(\cD)|\;Tv\vert_{\gG_1}=0\}, 
\end{equation*}
with norm
\begin{equation*}
||v||_V:=\left(\int_\cD|v|^2+||\grad v(x)||_2^2dx\right)^{1/2}.
\end{equation*}

\begin{rem}
	The condition $\gG_1\neq\emptyset$ implies that $V$ is a closed linear subspace of $H^1(\D)$.
	We may as well work with non-homogeneous boundary conditions on the Dirichlet part, i.e. $u(\go,x)=g_1(\go,x)$ for $g_1:\gO\times\gG_1\to\R$.
	The corresponding trace operator $T$ is still well defined if, for $\bP$-a.e. $\omega\in\Omega$, $g_1(\omega,\cdot)$ can be extended to a function $\widetilde g_1(\omega,\cdot)\in H^1(\D)$.
	Then, we consider for $\bP$-a.e. $\omega\in\Omega$ the problem
	\begin{align*}
	-\grad\cdot \left(a(\go,x)\grad ((u-\widetilde g_1)(\go,x))\right)&=f+\grad\cdot(a(\go,x)\grad\widetilde g_1(\go,x))\quad\text{on $\gO\times\cD$},\\
	(u-\widetilde g_1)(\go,x)&=0 \quad\text{on $\gG_1\times\cD$ and}\\
	a(\go,x)\vv n\cdot\grad((u-\widetilde g_1)(\go,x))&=g(\go,x)-a(\go,x)\vv n\cdot\grad\widetilde g_1(\go,x)\quad\text{on $\gG_2\times\cD$.}
	\end{align*}
	But this is in fact a version of Problem~\eqref{eq:elliptic} equipped with Eqs.~\eqref{eq:boundary} where the source term and Neumann-data have been changed (see also~\cite[p. 317]{E10}).
\end{rem}
As the coefficient and the boundary conditions are given by random functions, the solution $u$ is also a random function. Besides pathwise properties, $u$ may also have certain integrability properties with respect to the underlying probability measure.
To this end, we introduce the space of \textit{Bochner integrable} random variables resp. random functions (see~\cite{DZ14} for an overview).
\begin{defi}
	Let $(B,||\cdot||_B)$ be a Banach space and define the norm $||\cdot||_{L^p(\gO;B)}$ for a $B$-valued random variable $v:\gO\to B$ as
	\begin{equation*}
	||v||_{L^p(\gO, B)}:=\begin{cases}
	\E(||v||_B^p)^{1/p}\quad\text{for $1\le p<+\infty$}\\
	\text{esssup}_{\go\in\gO} ||v||_B\quad\text{for $p=+\infty$}
	\end{cases}.
	\end{equation*}
	The corresponding space of Bochner-integrable random variables is then given by
	\bee
	L^p(\gO;B):=\{w:\gO\to B\text{ is strongly measurable and }||w||_{L^p(\gO;B)}<+\infty\}.
	\eee
\end{defi}

The following set of assumptions on $a,f$ and $g$ allows us to show existence and uniqueness of the solution to Eq.~\eqref{eq:elliptic}. Consequently, we denote by $\cV'$ the topological dual of a vector space $\cV$.

\begin{assumption}\label{ass:unif_ell}
	Let $H:=L^2(\D)$. For $\bP$-almost all $\go\in\gO$ it holds that:
	\begin{itemize}
		\item $a_-(\go):=\inf_{x\in\cD} a(\go,x)>0$ and $a_+(\go):=\sup_{x\in\cD} a(\go,x)<\infty$.
		\item $1/a_-\in L^p(\gO;\R)$, $f\in L^q(\gO;H)$ and $g\in L^q(\gO;L^2(\gG_2))$ for some $p,q\in[1,\infty]$ such that $r:=(1/p+1/q)^{-1}\ge1$.
	\end{itemize}
\end{assumption}
\begin{rem}
	In Assumption~\ref{ass:unif_ell}, we did not establish a uniform elliptic bound on $a$ in $\gO$ (see for example \cite{LWZ16}), neither did we assume a certain spatial regularity as in \cite{CST13,G13}.
	The relatively weak assumptions are natural in our context, since in Section~\ref{sec:a_conv} we model $a$ as a jump-diffusion coefficient and uniform bounds or assumptions on H\"older-continuity are too restrictive.
	For the investigation of problem~\eqref{eq:elliptic} with piecewise H\"older-continuous coefficients we refer to \cite{G13}.
	We may identify $H$ with its dual and work on the Gelfand triplet $V\subset H\simeq H'\subset V'$.
	Hence, Assumption~\ref{ass:unif_ell} guarantees that $f(\go,\cdot)\in V'$, and, similarly, $g(\go,\cdot)\in H^{-1/2}(\gG_2)$ for $\bP$-almost all $\go\in\gO$.
\end{rem}

For fixed $\go\in\gO$, multiplying the random PDE~\eqref{eq:elliptic} with a test function $v\in V$ and integrating by parts yields the integral equation
\begin{equation}\label{eq:int_eq}
\int_\cD a(\go,x)\grad u(\go,x)\cdot\grad v(x)dx=\int_\cD f(\go,x)v(x)dx+\int_{\gG_2}g(\go,x)[Tv](x)dx.
\end{equation}
Consider the bilinear form $B_{a(\go)}$ 
\begin{equation*}
B_{a(\go)}:V\times V\to\R, \quad(u,v)\mapsto \int_\cD a(\go,x)\grad u(x)\cdot\grad v(x)dx
\end{equation*}
and 
\begin{equation*}
F_\go:V\to\R,\quad v\mapsto\int_\cD f(\go,x)v(x)dx+\int_{\gG_2}g(\go,x)[Tv](x)dx,
\end{equation*}
where the integrals in $F_\go$ are understood as the duality pairings 
\begin{align*}
\int_\cD f(\go,x)v(x)dx&=\dualpair{V'}{V}{f(\go,\cdot)}{v}\quad\text{and}\\
\int_{\gG_2} g(\go,x)[Tv](x)dx&=\dualpair{H^{-1/2}(\gG_2)}{H^{1/2}(\gG_2)}{g(\go,\cdot)}{Tv}.
\end{align*}
Equation~\eqref{eq:int_eq} then leads to the pathwise variational formulation of Problem~\eqref{eq:elliptic}:\\ 
For $\bP$-almost all $\go\in\gO$, given $f(\go,\cdot)\in V'$ and $g(\go,\cdot)\in H^{-1/2}(\gG_2)$, find $u(\go,\cdot)\in V$ such that 
\begin{equation}\label{eq:var}
B_{a(\go)}(u(\go,\cdot),v)=F_\go(v)
\end{equation}
for all $v\in V$. 
A function $u(\go,\cdot)\in V$ that fulfills the pathwise variational formulation is then called \textit{pathwise weak solution} to Problem~\eqref{eq:elliptic}.
\begin{thm}\label{thm:pw_u}
	If Assumption~\ref{ass:unif_ell} holds, then there exists a unique pathwise weak solution $u(\go,\cdot)\in V$ to Problem~\eqref{eq:var} for $\bP$-almost all $\go\in\gO$.
	Furthermore, $u\in L^r(\gO;V)$ and
	\begin{equation*}
	||u||_{L^r(\gO;V)}\le C(a_-,\D,p)(||f||_{L^q(\gO;H)}+||g||_{L^q(\gO;L^2(\gG_2))}),
	\end{equation*}
	where $C(a_-,\D,p)>0$ is a constant depending only on the indicated parameters.
\end{thm}
\begin{proof}
	Choose $\go\in\gO$ such that Assumption~\ref{ass:unif_ell} is fulfilled. For all $u,v\in V$, we obtain by the Cauchy-Schwarz inequality 
	\bee
	|B_{a(\go)}(u,v)|\le\left(\int_\D (a(\go,x))^2||\grad u(x)||_2^2dx\int_\D ||\grad v(x)||_2^2dx\right)^{1/2}
	\le a_+(\go)||u||_V||v||_V.
	\eee
	On the other hand,
	\begin{align*}
	B_{a(\go)}(u,u)&\ge a_-(\go)\int_\D ||\grad u(x)||_2^2dx\\
	&=\frac{a_-(\go)}{2}(||\grad u||_{L^2(\D)}^2+||\grad u||_{L^2(\D)}^2)\\
	&\ge\frac{a_-(\go)}{2}(||\grad u||_{L^2(\D)}^2+C_{|\D|}^{-2}||u||_{L^2(\D)}^2)\\
	&\ge\frac{a_-(\go)}{2}\min(1,C_{|\D|}^{-2})||u||_V^2,
	\end{align*}
	where the constant $C_{|\D|}^2>0$ stems from the constant in Poincar\'e's inequality, $C_{|\D|}$, and only depends on $|\D|$. 
	Hence the bilinear form $B_{a(\go)}:V\times V\to\R$ is continuous and coercive.
	We use that $H\simeq H'\subset V'$ and the trace theorem (Equation~\eqref{eq:tracebound}) to bound $F_\go$ by 
	\begin{align*}
	F_\go(v)&\le ||f(\go,\cdot)||_{V'}||v||_V+||g(\go,\cdot)||_{H^{-1/2}(\gG_2)}||Tv||_{H^{1/2}(\gG_2)}\\
	&\le (||f(\go,\cdot)||_H+C_\D||g(\go,\cdot)||_{L^2(\gG_2)})||v||_V
	\end{align*}
	This shows that $F_\go$ is a bounded linear functional on $V$ (and therefore continuous) for almost all $\go\in\gO$.
	The existence of a unique pathwise weak solution $u(\go,\cdot)$ is then guaranteed by the Lax-Milgram lemma $\bP$-almost surely. 
	If $u(\go,\cdot)$ is a solution of Eq.~\eqref{eq:var} for given $f(\go,\cdot)\in H$ and $g(\go,\cdot)\in L^2(\gG_2)$, then
	\begin{align*}
	\frac{a_-(\go)}{2}\min(1,C_{|\D|}^{-2})||u(\go,\cdot)||_V^2&\le B_{a(\go)}(u(\go,\cdot),u(\go,\cdot))\\
	&=F_\go(u(\go,\cdot))\le (||f(\go,\cdot)||_H+C_\D||g(\go,\cdot)||_{L^2(\gG_2)})||u(\go,\cdot)||_V
	\end{align*}
	Using H\"older's and Minkowski's inequality together with $r=(1/p+1/q)^{-1}\ge1$ yields:
	\begin{align*}
	||u||_{L^r(\gO;V)}&\le\frac{2\max(1,C_\D)}{\min(1,C_{|\D|}^{-2})}\E\left(a_-^{-p}\right)^{1/p}\E\left((||f||_H+||g||_{L^2(\gG_2)})^q\right)^{1/q}\\
	&\le \underbrace{\frac{2\max(1,C_\D)}{\min(1,C_{|\D|}^{-2})} ||1/a_-||_{L^p(\gO;\R)}}_{:=C(a_-,\D,p)}(||f||_{L^q(\gO;H)}+||g||_{L^q(\gO;L^2(\gG_2)})<+\infty.
	\end{align*}
\end{proof}

In the next section, we introduce the diffusion coefficient $a$, which allows us to incorporate discontinuities at random points or areas in $\D$. 
We show the existence and uniqueness of a weak solution to the discontinuous diffusion problem by choosing $a$ such that Assumption~\ref{ass:unif_ell} is fulfilled and Theorem~\ref{thm:pw_u} may be applied. 

\section{Discontinuous random elliptic problems}\label{sec:a_conv}
The stochastic coefficient $a$ in a jump-diffusion model should incorporate random discontinuities as well as a Gaussian component.  
We achieve this characteristic form of $a$ by defining the coefficient as a Gaussian random field with additive discontinuities on random areas of $\D$.  
Since this usually involves infinite series expansions in the Gaussian component or sampling errors in the jump measure, 
we also describe how to obtain tractable approximations of $a$.
Subsequently, existence and uniqueness results for weak solutions of the unapproximated resp. approximated jump-diffusion problem based on the results in Section~\ref{sec:setting} are proved.
We conclude this section by showing pathwise and $L^p$-convergence of the approximated solution to the solution $u:\gO\to V$ to the (unapproximated) discontinuous diffusion problem.

\subsection{Jump-diffusion coefficients and their approximations}
\begin{defi}\label{def:a}
	The \textit{jump-diffusion coefficient} $a$ is defined as
	\begin{equation}\label{eq:a}
	a:\gO\times\D\to\R_{>0},\quad (\go,x)\mapsto\ol a(x)+\Phi(W(\go,x))+P(\go,x),
	\end{equation}
	where
	\begin{itemize}
		\item $\ol a\in C^1(\D;\R_{\ge0})$ is non-negative, continuous and bounded.
		\item $\Phi\in C^1(\D;\R_{>0})$ is a continuously differentiable, positive mapping.
		\item $W\in L^2(\gO;H)$ is a (zero-mean) Gaussian random field associated to a non-negative, symmetric trace class operator $Q:H\to H$. 
		\item $\gl$ is a finite measure on $(\D,\cB(\D))$ and $\cT:\gO\to\cB(\D),\;\go\mapsto\{\cT_1,\dots,\cT_{\tau}\}$ is a random partition of $\D$ with respect to $\gl$.
		The number $\tau$ of elements in $\cT$ is a random variable $\tau:\gO\to\N$ on $(\gO,\cA,\bP)$ with $\E(\tau)=\gl(\D)$. 
		\item $(P_i, i\in\N)$ is a sequence of non-negative random variables on $(\gO,\cA,\bP)$ and
		$$P:\gO\times\D\to\R_{\ge0},\quad(\go,x)\mapsto\sum_{i=1}^{\tau(\go)}\1_{\{\cT_i\}}(x)P_i(\go).$$
		The sequence $(P_i, i\in\N)$ is independent of $\tau$ (but not necessarily i.i.d.).
	\end{itemize}
\end{defi}

\begin{rem}
	The definition of the measure $\gl$ on $(\D,\cB(\D))$ in Def.~\ref{def:a} relates not only to the average number of partition elements $\E(\tau)$, but may further be utilized to concentrate discontinuities of the jump-diffusion coefficient $a$ to certain areas of $\D$. 
	Choosing, for instance, $\gl$ as the Lebesgue measure on $\D$ corresponds to uniformly distributed jumps and on average equally sized partition elements $\cT_i$.
	In contrast, if $\gl$ is a Gaussian measure on $\D$ around some center point $x_C\in\D$, the number of discontinuities (resp. size of partition elements) will decrease (resp. increase) as one moves away from $x_C$. We refer to the numerical experiments in Section~\ref{sec:num}, where we give interpretations of $\gl$ to model certain characteristics of different jump-diffusion coefficients. On a further note, we do not require stochastic independence of $W$ and $P$.
\end{rem}
In general, the structure of $a$ as in Def.~\ref{def:a} does not allow us to draw samples from the exact distribution of this random function.
For an approximation of the Gaussian field, one usually uses truncated \KL expansions:
Let $((\eta_i,e_i), i\in\N)$ denote the sequence of eigenpairs of $Q$, where the eigenvalues are given in decaying order $\eta_1\ge\eta_2\ge\dots\ge0$.
Since $Q$ is trace class the Gaussian random field $W$ admits the representation
\bee
W=\sum_{i\in\N}\sqrt{\eta_i}e_iZ_i,
\eee
where $(Z_i,i\in\N)$ is a sequence of independent and standard normally distributed random variables.
The series above converges in $L^2(\gO;H)$ and $\bP$-almost surely (see i.e.~\cite{A15}).
The truncated \KL expansion $W_N$ of $W$ is then given by 
\bee
W_N:=\sum_{i=1}^N\sqrt{\eta_i}e_iZ_i,
\eee
where we call $N\in\N$ the \textit{cut-off index} of $W_N$.
In addition, it may be possible that the sequence of jumps $(P_i, i\in\N)$ cannot be sampled exactly but only with an intrinsic bias (see also Remark~\ref{rem:P}).
The biased samples are denoted by $(\widetilde P_i, i\in\N)$ and the error which is induced by this approximation is represented by the parameter $\eps>0$ as in Assumption~\ref{ass:EV}.
To approximate $P$ using the biased sequence $(\widetilde P_i, i\in\N)$ instead of $(P_i, i\in\N)$ we define the jump part approximation
\bee
P_\eps:\gO\times\D\to\R,\quad(\go,x)\mapsto\sum_{i=1}^{\tau(\go)}\1_{\{\cT_i\}}(x)\widetilde P_i(\go).
\eee
The \textit{approximated jump-diffusion coefficient} $a_{N,\eps}$ is then given by
\be \label{eq:a_approx}
a_{N,\eps}(\go,x):=\ol a(x)+\Phi(W_N(\go,x))+P_\eps(\go,x),
\ee
and the corresponding stochastic PDE with approximated jump-diffusion coefficient reads
\begin{align}\label{eq:u_N}
\begin{split}
-\grad \cdot\left(a_{N,\eps}(\go,x)\grad u_{N,\eps}(\go,x)\right)&=f(\go,x)\quad\text{in $\gO\times\cD$},\\
u_{N,\eps}(\go,x)&=0\quad\text{on $\gO\times\gG_1$},\\
a_{N,\eps}(\go,x)\vv n\cdot\grad u_{N,\eps}(\go,x)&=g(\go,x)\quad\text{on $\gO\times\gG_2$}.
\end{split}
\end{align}
% which corresponds to the original model problem where $a$ is replaced by $a_{N,\eps}$.
For $\go\in\gO$ and given samples $a_{N,\eps}(\go,\cdot), f(\go,\cdot)$ and $g(\go,\cdot)$, 
we consider the pathwise weak solution $u_{N,\eps}(\go,\cdot)\in V$ to Problem~\eqref{eq:u_N} for fixed approximation parameters $N\in\N$ and $\eps>0$.
The variational formulation of Eq.~\eqref{eq:u_N} is then analogous to Eq.~\eqref{eq:var} given by: 
For almost all $\go\in\gO$ with given $f(\go,\cdot)$, $g(\go,\cdot)$, find $u_{N,\eps}(\go,\cdot)\in V$ such that
\be
\begin{split}
	\label{eq:var_N}
	B_{a_{N,\eps}(\go)}(u_{N,\eps}(\go,\cdot),v):&=\int_\D a_{N,\eps}(\go,x)\grad u_{N,\eps}(\go,x)\cdot\grad v(x)dx\\
	&=\int_\D f(\go,x)v(x)dx+\int_{\gG_2}g(\go,x)[Tv](x)dx=F_\go(v)
\end{split}
\ee
for all $v\in V$.
The following assumptions guarantee that we can apply Theorem~\ref{thm:pw_u} also in the jump-diffusion setting and that therefore pathwise solutions $u$ and $u_{N,\eps}$ exist.

\begin{assumption}\label{ass:EV}
	~
	\begin{enum_i}
		\item The eigenfunctions $e_i$ of $Q$ are continuously differentiable on $\D$ and there exist constants $\ga,\gb,C_e,C_\eta>0$ such that for any $i\in\N$
		\bee
		||e_i||_{L^\infty(\D)}\le1,\quad||\grad e_i||_{L^\infty(\D)}\le C_ei^\alpha\quad\text{and}\quad\sum_{i=1}^\infty\eta_ii^\gb\leq C_\eta<+\infty.
		\eee 
		
		\item Furthermore, the mapping $\Phi$ as in Definition~\ref{def:a} and its derivative are bounded by
		\bee
		\Phi(w)\ge\phi_1\exp(-\psi_1 w^2),\quad|\Phi'(w)|\le\phi_2\exp(\psi_2|w|),\quad w\in\R,
		\eee
		where $0<\psi_1<(2\,tr(Q))^{-1}$ with $tr(Q):=\sum_{i\in\N}\eta_i$  and $\phi_1,\phi_2,\psi_2>0$ are arbitrary constants.
		\item Given $\psi_1$ and $tr(Q)$, there exists a $q>(1-2\,tr(Q)\psi_1)^{-1}=:(1-\eta^*)^{-1}\ge1$ such that $f\in L^q(\gO;H)$ and $g\in L^q(\gO;\gG_2)$.
		\item Finally, for some $\widetilde{s}\in[1,\infty)$, $(P_i, i\in\N)$ consists of $s$-integrable random variables, i.e. $P_i\in L^s(\Omega;\R_{\ge0})$ for all $i\in \N$ and $s\in[1,\widetilde{s}]$, further there exists a sequence of approximations $(\widetilde P_i, i\in\N)$ so that the sampling error is bounded, for $\eps>0$, by
		\begin{equation*}
		\E(|\widetilde P_i-P_i|^s)\le\eps,\quad i\in\N,\;s\in[1,\widetilde{s}].
		\end{equation*}
		
	\end{enum_i}

\end{assumption}
\begin{rem}\label{rem:P}
	Assumption~\ref{ass:EV}~(i) on the eigenpairs of $((\eta_i,e_i), i\in\N)$ is natural. 
	For instance, the case that $W$ is a Brownian-motion-like random field or that $Q$ is a Mat\'ern covariance operator are included. 
	The bounds on $\Phi$ and the regularity assumptions on $f$ and $g$ (Assumption~\ref{ass:EV}~(ii),(iii)) are necessary to ensure that the solution $u$ has at least finite expectation. 
	The sampling error $\E(|\widetilde P_i-P_i|^s)$ in Assumption~\ref{ass:EV}~(iv) may be interpreted in several ways:
	For instance, it may account for uncertainties in the distribution of $P_i$, like parameters for which only confidence intervals are available.
	Another possibility is, that realizations of $P_i$ may not be simulated directly or only at relatively high computational costs, for example by Acceptance Rejection algorithms, see \cite[Chapter II]{AG07}. 
	In this case, it may be favorable to generate the approximation $\widetilde P_i$ by a more efficient numerical algorithm (i.e. Fourier inversion techniques, see~\cite{BS17}) instead and to control for the error.
	The resulting sampling error can then be equilibrated with the truncation error from the Gaussian field to achieve a desired overall accuracy.
\end{rem}

\subsection{Existence of $u$ and $u_{N,\eps}$}
We first show the existence of a weak solution for both, the jump-diffusion problem~\eqref{eq:elliptic} with $a$ as in Definition~\ref{def:a}, and the approximated problem \eqref{eq:u_N}. 
\begin{lem}\label{lem:a}
	Let $a$ be a jump-diffusion coefficient (as defined in Def.~\ref{def:a}). 
	If $a,f$ and $g$ fulfill Assumptions~\ref{ass:EV}, then the elliptic problem \eqref{eq:elliptic} has a unique weak solution $u\in L^r(\gO;V)$, where $r\in[1,(1/q+\eta^*)^{-1})$ for $\eta^*:=2\,tr(Q)\psi_1$.
\end{lem}
\begin{proof}
	By Theorem~\ref{thm:pw_u}, it is sufficient to show that $a,f$ and $g$ fulfill Assumption~\ref{ass:unif_ell}.
	Clearly, $0<a_-(\go)\le a_+(\go)<\infty$ for almost all $\go\in\gO$, as $\ol a$ and $P$ are non-negative and we have the lower bound on $\Phi$ in Assumption~\ref{ass:EV}~(ii) by definition.
	Consequently, it is sufficient to bound the expectation of $(\inf_{x\in\D} \Phi(W(x))^p$, for $1\leq p<(\eta^*)^{-1}$, from below (see also~\cite[Section 2.3]{C12}).
	
	The random variable $W(x)-W(y)$ follows a centered normal distribution for any $x,y\in\D$.
	To see this, consider the finite sum 
	\bee
	\widetilde W_M(x,y):=\sum_{i=1}^M\sqrt{\eta_i}(e_i(x)-e_i(y))Z_i
	\eee
	for $M\in\N$ and a sequence $(Z_i, i=1,\ldots,M)$ of i.i.d. $\cN(0,1)${-distributed} random variables. 
	Clearly, $\widetilde W_M(x,y)$ is normally distributed with zero mean and characteristic function
	\bee
	\phi_M(t):=\E(\exp(it\widetilde W_M(x,y))=\exp(-\frac{t}{2}\sum_{i=1}^M\eta_i(e_i(x)-e_i(y))^2),\quad t\in\R.
	\eee
	Using that $Q$ is trace class with $|e_i(z)|\le1$ for all $i\in\N$ and $z\in\cD$, it follows
	\bee
	\gs^2(x,y):=\sum_{i=1}^\infty\eta_i(e_i(x)-e_i(y))^2\le 2tr(Q)<+\infty
	\eee
	and hence 
	\bee
	\lim_{M\to\infty}\phi_M(t)=\exp(-\frac{t}{2}\sum_{i=1}^\infty\eta_i(e_i(x)-e_i(y))^2),\quad t\in\R,
	\eee
	where the right hand side is the characteristic function of a normal distribution with zero mean and variance $\gs^2(x,y)$.
	By the L\'evy continuity theorem this implies that 
	\bee
	W(x)-W(y)\sim\cN(0,\gs^2(x,y))
	\eee 
	for all $x,y\in\cD$.
	Next we show that $W$ has $\bP$-almost surely H\"older continuous paths (see also~\cite[Proposition 3.4]{C12}):
	Let $0<b\le\min(1,\frac{\gb}{2\ga})$ (where $\ga,\gb$ are defined in Ass.~\ref{ass:EV}~(i)) and $Z\sim\cN(0,1)$. For $x,y,\in\D\subset\R^d$ and any $k\in\N$ we have 
	\begin{align*}
	\E(|W(x)-W(y)|^{2k})&=\E(|\sqrt{\gs^2(x,y)}Z|^{2k})\\
	&=\frac{(2k)!}{2^{k}k!}\left(\sum_{i>N}\eta_i(e_i(x)-e_i(y))^{2(1-b)+2b}\right)^{k}\\
	&\le \frac{(2k)!}{2^{k}k!}2^{2(1-b)k}C_e^{2bk}\left(\sum_{i>N}\eta_ii^{2\ga b}\right)^{k}||x-y||_2^{2bk},
	\end{align*}
	where $0<b<\gb/(2\ga)$ and the constant $C_e$ stems from Ass.~\ref{ass:EV}~(i).
	The second equality results from the fact that $\E(Z^{2k})=(2k)!/(2^{k}k!)$ for all $k\in\N$
	and the sum in the inequality is finite because $2\ga b<\gb$.
	For any dimension $d\in\N$, we may choose $k>d/(2b)$ and obtain by the Kolmogorov-Chentsov theorem (\cite[Theorem 3.5]{DZ14}) that $W$ has a H\"older continuous modification with H\"older exponent $\gg\in(0,(2bk-d)/2k)$.
	Hence, $W$ is a centered Gaussian process and almost surely bounded on $\D$.
	By \cite[Theorem 2.1.1]{AT09} this implies 
	$
	E:=\E(\sup_{x\in\D} W(x))<\infty
	$
	and 
	\be\label{eq:tail}
	\bP(\sup_{x\in\D} W(x)-E\ge c)\le \exp(-\frac{c^2}{2\ol\sigma^2})
	\ee
	for all $c>0$ and $\overline\sigma^2:=\sup_{x\in\D}\E(W(x)^2)$.
	With Assumption~\ref{ass:EV}~(ii) and since $||\exp(|W|)||_{L^\infty(\D)}\le \exp(||W||_{L^\infty(\D)})$ we further obtain
	\begin{align*}
	\E(1/a_-^p)&\le\E((\inf_{x\in\D}\Phi(W(\cdot,x))^{-p})=\E(\sup_{x\in\D}\Phi(W(x))^{-p})\\
	&\le \frac{1}{\phi_1^p}\E(\sup_{x\in\D}\exp(p\psi_1|W(x)|^2))\\
	&\le \frac{1}{\phi_1^p}\E(\exp(p\psi_1||W||_{L^\infty(\D)}^2)).
	\end{align*}
	By Fubini's Theorem and integration by parts we may bound
	\begin{align*}
	\E(\exp(p\psi_1||W||_{L^\infty(\D)}^2))&=\int_0^\infty 2p\psi_1 c\exp(p\psi_1c^2)\bP(||W||_{L^\infty(\D)}>c)dc\\
	&\le 2p\psi_1 E\exp(p\psi_1E^2) + 2\int_E^\infty p\psi_1 c\exp(p\psi_1c^2))\bP(||W||_{L^\infty(\D)}>c)dc\\
	&\le 2p\psi_1 E\exp(p\psi_1E^2) + 4\int_E^\infty p\psi_1 c\exp((p\psi_1-\frac{1}{2\ol\gs^2})c^2))dc,
	\end{align*} 
	where we have used $\bP(||W(x)||_{L^\infty(\D)}>c)\le 2\bP(\sup_{x\in\D} W(x)>c)$, by the symmetry of $W$, and
	Ineq.~\eqref{eq:tail} in the last step. 
	The last expectation is finite if and only if $p\psi_1<\frac{1}{2\overline\sigma^2}$.
	For this to hold, $p<(\eta^*)^{-1}$ is sufficient, since $W(x)\sim N(0,\sum_{i=1}^\infty\eta_i e_i(x))$ and thus $\overline\sigma^2\le tr(Q)$.
\end{proof}

\begin{rem}\label{rem:regularity_u}
	{From Lemma~\ref{lem:a} follows immediately that one cannot expect finite second moments of the solution $u$ for $\eta^*\geq 1/2$ or $q\leq 2$. If we assume that $q=3$ we need $\eta^*<1/6$ to have finite second moments (in the case of, for instance, a log-normal Gaussian field). Note that, for all covariance kernels and functionals we use (e.g. log-Gaussian fields with Mat\'ern class covariance or Brownian-motion like covariance kernels), $\psi_1$ and $\eta^*$ are much smaller than $1/2$. If one assumes that $f$, $g$ and $a$ are stochastically independent, then the regularity of the solution (in $\Omega$) is at least the same as the lowest regularity of the data, i.e. $f$, $g$ or $a$.}
	%  \att{{This might need some further formulation....}}
	%  \att{\as{Maybe we should also comment why $\psi_1$ resp. $\eta^*$ is in general rather small, i.e. for a log-Gaussian field.}}
\end{rem}

\begin{lem}\label{lem:a_n}
	Let $a_{N,\eps}$ be the approximated jump-diffusion coefficient (as in Eq.~\eqref{eq:a_approx}) and define the random variables 
	\bee
	a_{N,\eps,-}:\gO\to\R,\quad\go\mapsto\inf_{x\in\D}a_{N,\eps}(\go,x)\quad\text{and}\quad a_{N,\eps,+}:\gO\to\R,\quad\go\mapsto\sup_{x\in\D}a_{N,\eps}(\go,x)
	\eee
	on $(\gO,\cA,\bP)$.
	If Assumption~\ref{ass:EV} holds, then, for any $N\in\N$ and $\eps>0$, there exists a unique weak solution $u_{N,\eps}\in L^r(\gO;V)$ to Problem~\eqref{eq:u_N}, where $r\in[1,(1/q+\eta^*)^{-1})$, for $\eta^*:=2\,tr(Q)\psi_1$.
	Furthermore, $||1/a_{N,\eps,-}||_{L^p(\gO,\R)}$ is bounded uniformly with respect to $\eps$ and $N$ for $p\in[1,(\eta^*)^{-1})$.
\end{lem} 
\begin{proof}
	The proof is carried out identically to Lemma~\ref{lem:a}, where we replace $a_-$ by $a_{N,\eps,-}$, $a_+$ by $a_{N,\eps,+}$, $\gs(x,y)^2$ by $\sum_{i=1}^N\eta_i(e_i(x)-e_i(y))^2$ and $tr(Q)$ by $\sum_{i=1}^N\eta_i$.
	Again, by Eq.~\eqref{eq:a_approx}, $a_{N,\eps,+}<+\infty$ $\bP$-almost surely.
	In the case that $\eta_1=0$, the random field $W$ is degenerated and equal to zero.
	Hence, $a_{N,\eps}(\go,x)\ge\phi_1>0$ for all $(\go,x)\in\gO\times\D$ and the claim follows immediately.
	Otherwise, if $\eta_1>0$, we obtain in the fashion of Lemma~\ref{lem:a}
	\begin{align*}
	\E(1/a_{N,\eps-}^p)&\le \frac{1}{\phi_1^p}\E\big(\exp(p\psi_1||W_N||_{L^\infty(\D)}^2)\big)\\
	&\le \frac{1}{\phi_1^p}\E\big(\exp(p\psi_1||W||_{L^\infty(\D)}^2)\big)
	\end{align*}
	which is again finite if $p<(\eta^*)^{-1}$. 
	The proof is concluded by noting that the last estimate is independent of $N$ and $\eps$.
\end{proof}

\subsection{Convergence of the approximated diffusion coefficient}
The convergence of the approximated solution $u_{N,\eps}$ to $u$ depends on the convergence of the approximated jump-diffusion coefficient $a_{N,\eps}$ to $a$. 
We investigate this convergence by deriving separately convergence rates of the truncated \KL series and the approximation of the jump part.
\begin{thm}\label{thm:w_error}
	If Assumption~\ref{ass:EV} holds, then for any $p\ge1$ and $N\in\N$ we have that 
	\bee
	||W-W_N||_{L^p(\gO;L^\infty(\cD))}\le C_p\Xi_N^{1/2},
	\eee
	where $C_p>0$ is independent of $N$ and $\Xi_N:=\sum_{i>N}\eta_i<\infty$.
\end{thm}
In order to prove the bound for the truncation error, we need the following Fernique-type result. 
\begin{thm}\cite[Theorem 2.9]{LS01}\label{thm:tail}
	Let $(W(x),x\in \D)$ with $\D\subset\R^d$ be a centered Gaussian field. For $\epsilon, \delta>0$ let 
	\bee
	\varrho(\epsilon):=\sup_{x\in\D,\;||x-y||_2<\epsilon} \E((W(x)-W(y))^2)^{1/2}
	\quad\text{ and }\quad
	\Theta(\gd):=\int_0^\infty \varrho(\gd \exp(-y^2))dy.
	\eee
	Then for all $c>0$ 
	\bee
	\bP(\sup_{x\in\D} W(x)>c)\le C (\Theta^{-1}(1/c))^{-d}\exp(-\frac{c^2}{2\overline{\sigma}}),
	\eee
	where $\Theta^{-1}$ is the inverse function of $\Theta$, $C>0$ is an absolute constant and $\ol\sigma:=\sup_{x\in\D}\E(W(x)^2)^{1/2}$.
\end{thm}

\begin{proof}[Proof of Theorem 3.8]
	With a similar continuity argument as in Lemma~\ref{lem:a}, we obtain that 
	\bee
	(W-W_N)(x)-(W-W_N)(y)\sim\cN(0,\sigma_N^2(x,y))
	\eee 
	where $\sigma_N^2(x,y):=\sum_{i>N} \eta_i(e_i(x)-e_i(y))^2$ for all $x,y,\in\D$. 
	Hence 
	\bee
	\ol\sigma^2_N:=\sup_{x\in\D}\E((W-W_N)(x)^2)=\sum_{i>N}\eta_ie_i(x)^2\le \sum_{i>N}\eta_i=\Xi_N
	\eee
	and we see by the proof of Lemma~\ref{lem:a} that for $\varrho$ as in Theorem~\ref{thm:tail} 
	$$\varrho(\epsilon)\le 2^{1-b}C_e^b\left(\sum_{i\in\N}\eta_ii^{2\ga b}\right)^{1/2}\epsilon^b =:C_{e,b} \epsilon^b$$ 
	for any $b\in (0,\min(\gb/(2\ga),1))$ and each $\epsilon>0$.
	We use the estimate on $\varrho$ to bound $\Theta$ for $\gd>0$ by
	\bee
	\Theta(\gd)\le C_{e,\ga}\int_0^\infty \gd^b \exp(-y^2b)dy=\frac{C_{e,b}\gd^b\sqrt{\pi}}{2\sqrt b}.
	\eee
	Since $\Theta$ and $\Theta^{-1}$ are monotone increasing this yields with Theorem~\ref{thm:tail} for any $c>0$
	\bee
	\bP(\sup_{x\in\D}(W-W_N)(x)>c)\le C \left(\frac{2\sqrt b}{C_{e,b}\sqrt{\pi}}\right)^{-d/b}c^{d/b}\exp(-\frac{c^2}{2\ol\gs^2_N}):=C_{e,\ga,b}\,c^{d/b}\exp(-\frac{c^2}{2\ol\gs^2_N}).
	\eee
	Using again that $\bP(\sup_{x\in\D}|(W-W_N)|(x)>c)\le 2\bP(\sup_{x\in\D}(W-W_N)(x)>c)$ and Fubini's Theorem, we have for any $p\ge1$
	\begin{align*}
	\E(||W-W_N||^p_{L^\infty(\D)})&=\int_0^\infty pc^{p-1}\bP(\sup_{x\in\D}|(W-W_N)(x)|>c)dc\\
	&\le 2 C_{e,\ga,b}\, p\int_0^\infty c^{p-1+d/b}\exp(-\frac{c^2}{2\ol\gs^2_N})dc\\
	&= 2^{1+p/2+d/(2b)}C_{e,\ga,b}\, p\gG((p+d/b)/2) \ol\gs_N^{p+d/b}\\
	&\le 2^{1+p/2+d/(2b)}C_{e,\ga,b}\, p\gG((p+d/b)/2) tr(Q)^{d/b}\Xi_N^{p/2},
	\end{align*}
	where $\gG(\cdot)$ is the Gamma function and we have used the substitution $z=c^2/(2\ol\gs_N^2)$ in the second equality. 
	This proves the claim because the above estimate holds for any $b\in (0,\min(\gb/(2\ga),1))$.
\end{proof}
\begin{rem}
	In \cite{C12}, the author proves a similar error bound for the truncation error in the Gaussian field, namely 
	\bee
	\E(||W-W_N||_{L^p(\gO;C^{0,\gg}(\cD))})\le C_{b,\gg,p}\max\left(\sum_{i>N}\eta_i\,,\,C_e^{2b}\sum_{i>N}\eta_ii^{2b}\right)^{1/2}.
	\eee
	In the jump-diffusion setting the error bound in Theorem~\ref{thm:w_error} is advantageous for several reasons:
	\begin{itemize}
		\item In our setting, $a$ will in general not have H\"older continuous paths, but involves discontinuities, hence only the error in $L^p(\gO;L^\infty(\cD))$ is of interest. 
		\item In general it is rather easy to calculate the sum $\Xi_N$ if the first $N$ eigenvalues are known (or can be approximated), whereas this is not necessarily the case for  $\sum_{i>N}\eta_ii^{2b}$.
	\end{itemize}
\end{rem} 

\begin{thm}\label{thm:p_error}
	Under Assumption~\ref{ass:EV}, the sampling error is, for all $s\in[1,\widetilde{s}]$, bounded by
	\bee
	||P-P_\eps||_{L^s(\gO;L^\infty(\cD))}\le (\gl(\D)\eps)^{1/s}.
	\eee
\end{thm}
\begin{proof}
	For any $\go\in\gO$, we have
	\bee
	||P(\go,\cdot)-P_\eps(\go,\cdot)||_{L^\infty(\D)}=\max_{i=1,\dots,\tau(\go)}|P_i(\go)-\widetilde P_i(\go)|^s.
	\eee
	By Fubini's Theorem and integration by parts this yields
	\bee
	||P-P_\eps||^s_{L^s(\gO;L^\infty(\D))}=\int_0^\infty sc^{s-1}\bP\left(\max_{i=1,\dots,\tau}|P_i-\widetilde P_i|\ge c\right)dc,
	\eee
	since $\lim\limits_{c\to+\infty}\bP(\max_{i=1,\dots,\tau}|P_i-\widetilde P_i|\ge c) = 0$. For fixed $c>0$ and $i\in\N$, we define the sets 
	\bee
	T_i:=\{\go\in\gO|\tau(\go)=i\}\quad\text{and}\quad A_i:=\{\go\in\gO|\,|P_i(\go)-\widetilde P_i(\go)|\ge c\}
	\eee
	to obtain the identity
	\bee
	\bP\left(\max_{i=1,\dots,\tau}|P_i-\widetilde P_i|\ge c\right)=\bP\left(\bigcup_{i\in\N}T_i\cap(\bigcup_{j=1}^iA_j)\right).
	\eee
	By the independence of $|P_i-\widetilde P_i|$ and $\tau$ this yields
	\begin{align*}
	\bP\left(\max_{i=1,\dots,\tau}|P_i-\widetilde P_i|\ge c\right)
	\le\sum_{i\in\N}\bP\left(T_i\cap(\bigcup_{j=1}^iA_j)\right)=\sum_{i\in\N}\bP(T_i)\bP(\bigcup_{j=1}^iA_j)\le\sum_{i\in\N}\bP(T_i)\sum_{j=1}^i\bP(A_{j})
	\end{align*}
	and thus by Fubini's theorem
	\begin{align*}
	||P-P_\eps||^s_{L^s(\gO;L^\infty(\D))} & = \int_0^\infty \int_0^y sc^{s-1}dc \bP(\|P-P_\eps\|_{L^\infty(\D)}\in dy)\\
	&= \int_0^\infty s c^{s-1} \bP(\|P-P_\eps\|_{L^\infty(\D)}\ge c)\,dc\\
	&\le\int_{0}^\infty\sum_{i\in\N}\bP(\tau=i)\sum_{j=1}^isc^{s-1}\bP(|P_j-\widetilde P_j|\ge c)\,dc\\
	&=\sum_{i\in\N}\bP(\tau=i)\sum_{j=1}^i\int_0^\infty sc^{s-1}\bP(|P_j-\widetilde P_j|\ge c)\,dc\\
	&=\sum_{i\in\N}\bP(\tau=i)\sum_{j=1}^i\E(|P_j-\widetilde P_j|^s)\\
	&\le\sum_{i\in\N}\bP(\tau=i)i\eps.
	\end{align*}
	The claim then follows since $\E(\tau)=\gl(\D)<+\infty$ by Definition~\ref{def:a}.\\
\end{proof}

With Theorems~\ref{thm:w_error} and \ref{thm:p_error}, follows convergence of the approximated diffusion coefficient:
\begin{thm}\label{thm:a_error}
	Let Assumption~\ref{ass:EV} hold, then there exists a constant $C(s,\Phi,Q,\gl,\D)>0$, depending only on the indicated parameters, such that for any $N\in\N$ and $\eps>0$
	\bee
	||a-a_{N,\eps}||_{L^s(\gO;L^\infty(\D))}\le C(s,\Phi,Q,\gl,\D)\left(\Xi_N^{1/2}+\eps^{1/s}\right).
	\eee
	Hence, $a_{N,\eps}$ converges to $a$ in $L^s(\gO;L^\infty(\D))$ as $N\to\infty$ and $\eps\to 0$.
\end{thm}
\begin{proof}
	Let $N\in\N, \eps>0$ and $(\go,x)\in\gO\times\D$ be fixed. 
	By the mean-value theorem
	\bee
	a(\go,x)-a_{N,\eps}(\go,x)=\Phi'(\xi_N(\go,x))(W(\go,x)-W_N(\go,x))+P(\go,x)-P_\eps(\go,x),
	\eee
	where $\xi_N(\go,x)\in(W(\go,x),W_N(\go,x))$.
	With Assumption~\ref{ass:EV}~(ii) on $\Phi'$ and by the triangle inequality, we obtain the estimate
	\bee
	|\Phi'(\xi_N(\go,x))|\le\phi_2\exp(\psi_2(|W_N(\go,x)|+|W(\go,x)-W_N(\go,x)|)).
	\eee
	The random fields $W_N$ and $W-W_N$ are independent, so for any $s\in[1,\infty)$ it holds that  
	\begin{equation}\label{eq:aw}
	\begin{split}
	&||\exp(\psi_2(|W_N|+|W-W_N|))(W-W_N)||_{L^s(\gO;L^\infty(\D))}\\
	&\quad\le||\exp(\psi_2|W_N|)||_{L^s(\gO;L^\infty(\D))}||\exp(\psi_2|W-W_N|^2)||_{L^{p_1}(\gO;L^\infty(\D))}||(W-W_N)||_{L^{p_2}(\gO;L^\infty(\D))}
	\end{split}
	\end{equation}
	where we have used H\"older's inequality with $p_1,p_2>s$ such that $1/s=1/p_1+1/p_2$.
	With Young's inequality we obtain
	\begin{align*}
	||\exp(\psi_2|W-W_N|)||_{L^{p_1}(\gO;L^\infty(\D))}^{p_1}&\le\E(\exp(\psi_2p_1||W-W_N||_{L^\infty(\D)})\\
	&\le\exp(\psi_2^2p_1^2 tr(Q))\E\Big(\exp\big(\frac{||W-W_N||^2_{L^\infty(\D)}}{4tr(Q)}\big)\Big)\\
	&=:C_2(p_1,\psi_2,Q).
	\end{align*}
	Since $(W-W_N)(\cdot,x)\sim\cN(0,\gs^2_N(x))$ with $\gs^2_N(x):= \sum_{i>N}\eta_ie_i(x)^2\le tr(Q)$, we proceed as in Lemma~\ref{lem:a}
	and conclude by Fernique's theorem that $C_2(p_1,\Phi,Q)<\infty$ and note that $C_2(p_1,\Phi,Q)$ is independent of $N$.
	Similarly, 
	\begin{align*}
	||\exp(\psi_2|W_N|)||_{L^s(\gO;L^\infty(\D))}^s
	&\le\exp(\psi_2^2s^2 tr(Q))\E\Big(\exp\big(\frac{||W_N||^2_{L^\infty(\D)}}{4tr(Q)}\big)\Big):=C_1(s,\psi_2,Q)<\infty
	\end{align*}
	for any $s$ and $C_1(s,\psi_2,Q)$ is independent of $N$. 
	Altogether, this yields 
	\bee
	||a-a_{N,\eps}||_{L^s(\gO;\R)}\le \phi_2C_1(s,\psi_2,Q))^{1/s}C_2(p_1,\psi_2,Q)^{1/p_1}||W-W_N||_{L^{p_2}(\gO;L^\infty(\D)}+||P-P_\eps||_{L^s(\gO;L^\infty(\D))}
	\eee
	and the claim follows by Theorems~\ref{thm:w_error} and \ref{thm:p_error}.
\end{proof}
\begin{cor}\label{cor:a_as}
	Let Assumption~\ref{ass:EV} hold, then there exists a sequence of approximation parameters $((N_i,\eps_i), i\in\N)$ in $\N^\N\times(0,\infty)^\N$, depending only on $\gb$ and $s$, such that the error $||a-a_{N_i,\eps_i}||_{L^\infty(\D)}$ converges to zero $\bP$-almost surely as $i\to\infty$.
\end{cor}
\begin{proof}
	For any $\epsilon>0$ we get by Markov's inequality 
	\bee
	\bP(||a_{N_i,\eps_i}-a||_{L^\infty(\D)}\ge\epsilon)\le\frac{||a_{N_i,\eps_i}-a||^s_{L^s(\gO;L^\infty(\D))}}{\epsilon^s}.
	\eee
	Using Theorem~\ref{thm:a_error} and the inequality $(a+b)^s\le 2^{s-1}(a^s+b^s)$ for $a,b>0,s\ge1$ this leads to
	\bee
	\sum_{i\in\N}\bP(||a_{N_i\eps_i}-a||_{L^\infty(\D)}\ge\epsilon)\le\frac{2^{s-1}C(s,\psi_2,Q,\gl,\D)^s}{\epsilon^s}\left(\sum_{i\in\N}\Xi_{N_i}^{s/2}+\eps_i\right).
	\eee
	By Assumption~\ref{ass:EV}, there exists $\gb>0$ such that
	\bee
	\Xi_{N_i}=\sum_{j>N_i}\eta_jj^\gb j^{-\gb}\le N_i^{-\gb}\sum_{j>N_i}\eta_jj^\gb\le N_i^{-\gb}\sum_{j\in\N}\eta_jj^\gb\le C_\eta N_i^{-\gb}.
	\eee
	Now, choosing $\gd>2$, $N_i:=\ceil{i^{\gd/\gb s}}$ and $\eps_i:=i^{-\gd/2}$ for $i\in\N$ yields the estimate
	\bee
	\sum_{i\in\N}\bP(||a_{N_i,\eps_i}-a||_{L^\infty(\D)}\ge\epsilon)\le\frac{2^{s-1}C(s,\psi_2,Q,\gl,\D)^s}{\epsilon^s}\left(C_\eta^{s/2}+1\right)\sum_{i\in\N}i^{-\gd/2}<\infty,
	\eee
	and the sequence $(a_{N_i,\eps_i}, i\in\N)$ converges almost surely by the Borel-Cantelli lemma.
\end{proof}

\subsection{Convergence of the approximated solution.}
We conclude this section by showing the convergence of $u_{N,\eps}$ towards $u$, given that $a_{N,\eps}\to a$ in $L^s(\gO;L^\infty(\cD))$.

\begin{thm}\label{thm:u_error}
	Let Assumptions~\ref{ass:EV} hold such that $t:=(2/p+1/q+1/s)^{-1}\ge1$, where $p\in[1, (\eta^*)^{-1})$ for $\eta^*$ as in Lemma~\ref{lem:a}.
	Then $u_{N,\eps}$ converges to $u$ in $L^t(\gO;V)$ as $N\to+\infty$ and $\eps\to0$.
\end{thm}
\begin{proof}
	Existence and uniqueness of weak solutions $u\in L^r(\gO;V)$ resp. $u_{N,\eps}\in L^r(\gO;V)$, for $r\in[1,(1/q+\eta^*)^{-1})$, is guaranteed by Lemma~\ref{lem:a} resp. Lemma~\ref{lem:a_n} $\bP$-almost surely.
	For notational convenience, we omit the argument $\go\in\gO$ in the following pathwise estimates with respect to the norm $||\cdot||_V$.
	With Poincar\'e's inequality and $a_{N,\eps,-}(\go)>0$ $\bP$-a.s., we obtain the (pathwise) estimate
	\begin{align*}
	||u-u_{N,\eps}||_V^2&\le\frac{1}{a_{N,\eps,-}}\int_\D a_{N,\eps}(|u-u_{N,\eps}|^2+||\grad u-\grad u_{N,\eps}||_2^2)dx\\
	&\le\frac{1+C_{|\D|}^2}{a_{N,\eps,-}}\int_\D a_{N,\eps}||\grad u-\grad u_{N,\eps}||_2^2dx
	\end{align*}
	where $C_{|\D|}>0$ denotes the Poincar\'e constant. 
	Since $u$ and $u_{N,\eps}$ are weak solutions of Problem~\eqref{eq:elliptic} with Eq.~\eqref{eq:boundary} resp. Eq.~\eqref{eq:u_N}, we have 
	\bee
	\int_\D a_{N,\eps}\grad u\cdot\grad u_{N,\eps}dx=\int_\D a||\grad u||_2^2dx
	\quad
	\text{and}
	\quad
	\int_\D a_{N,\eps}||\grad u_{N,\eps}||_2^2dx=\int_\D a\grad u\cdot\grad u_{N,\eps}dx,
	\eee
	almost surely, and hence
	\bee
	\int_\D a_{N,\eps}||\grad u-\grad u_{N,\eps}||_2^2dx=\int_\D (a_{N,\eps}-a)\grad u\cdot(\grad u-\grad u_{N,\eps})dx.
	\eee
	By H\"older's inequality, $V\subset L^2(\cD)$, Eq.~\eqref{eq:tracebound} and Theorem~\ref{thm:pw_u} we have
	\begin{align*}
	||u-u_{N,\eps}||_V^2&\le\frac{1+C_{|\D|}^2}{a_{N,\eps,-}}||a_{N,\eps}-a||_{L^\infty(\D)}||\grad u||_{L^2(\D)}||\grad (u-u_{N,\eps})||_{L^2(\D)}\\
	&\le\frac{1+C_{|\D|}^2}{a_{N,\eps,-}}||a_{N,\eps}-a||_{L^\infty(\D)} ||u||_V ||u-u_{N,\eps}||_V\\
	&\le\frac{2(1+C_{|\D|}^2)\max(1,C_\D)(||f||_{H}+||g||_{L^2(\gG_2)})}{\min(1,C_{|\D|}^{-2})a_{N,\eps,-}a_-} ||a_{N,\eps}-a||_{L^\infty(\D)} ||u-u_{N,\eps}||_V.
	\end{align*}
	Using $t=(2/p+1/q+1/s)^{-1}\ge1$, the $t$-th moment of the pathwise error is bounded by 
	\be\label{eq:u_error}
	\begin{split}
		||u-u_{N,\eps}||_{L^t(\gO;V)}\le&(1+C_{|\D|}^2)C(a_-,\D,p)||1/a_{N,\eps,-}||_{L^p(\gO;\R)}\\
		&\cdot(||f||_{L^q(\gO;H)}+||g||_{L^q(\gO;L^2(\gG_2)})  ||a_{N,\eps}-a||_{L^s(\gO;L^\infty(\D))},
	\end{split}
	\ee
	where $C(a_-,\D,p)>0$ is as in Theorem~\ref{thm:pw_u}. 
	The convergence now follows by Theorem~\ref{thm:a_error},
	since $||a_{N,\eps,-}^{-1}||_{L^p(\gO;\R)}$ is bounded uniformly in $N$ and $\eps$ by Lemma~\ref{lem:a_n}.
\end{proof}

\begin{cor}
	With the Assumptions of Theorem~\ref{thm:u_error}, there exists a sequence of approximation parameters $((N_i,\eps_i), i\in\N)$ in $\N^\N\times(0,\infty)^\N$, depending only on $\gb$ and $t$, such that $||u_{N_i,\eps_i}-u||_V$ converges to zero $\bP$-almost surely as $i\to+\infty$.
\end{cor}
\begin{proof}
	The proof is analogous to the one of Corollary~\ref{cor:a_as} with Markov's inequality applied to the mapping $x\mapsto x^t$ and Ineq.~\eqref{eq:u_error}.
	The sequence $(N_i,\eps_i, i\in\N)$ may then also be constructed in the same way as in Corollary~\ref{cor:a_as} where we simply replace $s$ by $t$.
\end{proof}

Knowing that $u_{N,\eps}$ converges to $u$ pathwise and in $L^t(\gO;V)$, we aim to estimate moments of $u$ by drawing samples from the distribution of $u_{N,\eps}$.
In general, the distribution of $u_{N,\eps}$ is not known. Further, each pathwise solution is an element of the infinite-dimensional Hilbert space $V$, which in turn means that we are only able to simulate pathwise approximations of the functions $u_{N,\eps}(\go,\cdot)$ in a finite-dimensional subspace of $V$. 
Next, we show how to construct these approximations in some appropriate subspaces $V_\ell$ of $V$ and how the discretization error may be controlled.

\section{Adaptive pathwise discretization}\label{sec:fem}
The variational problem~\eqref{eq:var_N} to Eq.~\eqref{eq:u_N} is to find for almost all $\go\in\gO$ and given $f(\go,\cdot)$, $g(\go,\cdot)$, $N$ and $\eps$ a function $u_{N,\eps}(\go,\cdot)\in V$ such that
\bee\tag{\ref{eq:var_N}}
\begin{split}
	B_{a_{N,\eps}(\go)}(u_{N,\eps}(\go,\cdot),v):&=\int_\D a_{N,\eps}(\go,x)\grad u_{N,\eps}(\go,x)\cdot\grad v(x)dx\\
	&=\int_\D f(\go,x)v(x)dx+\int_{\gG_2}g(\go,x)[Tv](x)dx=F_\go(v)\\
\end{split}
\eee
for all $v\in V$.
To find suitable approximations of $u_{N,\eps}(\go,\cdot)$, we use a standard Galerkin approach with a sequence $\cV=(V_\ell, \ell\in\N_0)$ of finite-dimensional subspaces $V_\ell\subset V$.
The corresponding family of refinement sizes is given by a sequence $(h_\ell, \ell\in\N_0)$, which decreases monotonically to zero as $\ell\to\infty$.
For any $\ell\in\N$, let $d_\ell\in\N$ and $\{v_1,\dots,v_{d_\ell}\}$ be a basis of $V_\ell$.  
The discrete version of the variational formulation~\eqref{eq:var_N} is then: Find $u_{N,\eps,\ell}(\go,\cdot)\in V_\ell$ such that 
\be
B_{a_{N,\eps}(\go)}(u_{N,\eps,\ell}(\go,\cdot),v_\ell)=F_\go(v_\ell)\quad\text{for all $v_\ell\in V_\ell$.}
\ee
The function $u_{N,\eps,\ell}(\go,\cdot)$ may be expanded with respect to the basis $\{v_1,\dots,v_{d_\ell}\}$ as
\bee
u_{N,\eps,\ell}(\go,x)=\sum_{i=1}^{d_\ell}c_iv_i(x),
\eee
where $c_1,\dots,c_{d_\ell}\in\R$ and ${\bf c} := (c_1,\dots,c_{d_\ell})^T$ is the respective coefficient (column-)vector.
With this, the discrete variational problem in the finite-dimensional space $V_\ell$ is equivalent to solving the linear system of equations
\bee
\mathbf A(\go) {\bf c}
=\mathbf F(\go),
\eee
with stochastic stiffness matrix $(\mathbf A(\go))_{ij}=B_{a_{N,\eps}(\go)}(v_i,v_j)$ and load vector $(\mathbf F(\go))_i=F_\go(v_i)$ for $i,j\in\{1,\dots,d_\ell\}$.
Since the jump-diffusion coefficient is not continuous, in general one would not expect the full convergence rate of the Galerkin approximation. 

\begin{ex}\label{ex:fem}
	For a polygonal domain $\D\subset\R^d$, we define by $\cK=(\cK_\ell, \ell\in\N_0)$ a sequence of triangulations on $\D$.
	We denote the minimum interior angle of all triangles in $\cK_\ell$ by $\vartheta_\ell>0$ and assume that there exists some $\vartheta>0$ such that 
	% 	\bee
	$\inf_{\ell\in\N_0}\vartheta_\ell\ge\vartheta>0$.
	% 	\eee
	The maximum diameter of each triangulation is defined by
	\bee
	h_\ell:=\max_{K\in\cK_\ell}\text{diam}(K),\quad\ell\in\N_0
	\eee 
	and the finite-dimensional subspaces are given by 
	\bee
	V_\ell:=\{v\in V|\; v|_K\in\cP_1,K\in\cK_\ell\},
	\eee
	where $\cP_1$ is the space of all polynomials up to degree one.
	This yields a sequence $\cV=(V_\ell, \ell\in\N_0)$ of subspaces in $V$ with refinement parameters $(h_\ell, \ell\in\N_0)$. 
	For fixed $\ell\in\N_0$, let $\{v_1,\dots,v_{d_\ell}\}$ be a basis of piecewise linear functions of $V_\ell$.
	Given that $u_{N,\eps}\in L^2(\gO;H^{1+\gk_a}(\D))$ for some $\gk_{a}>0$, the pathwise discretization error is bounded by C\'ea's lemma $\bP$-almost surely by
	\begin{align*}
	||u_{N,\eps}(\go,\cdot)-u_{N,\eps,\ell}(\go,\cdot)||_V
	&\le (1+C_{|\D|}) ||\nabla(u_{N,\eps}(\go,\cdot)-u_{N,\eps,\ell}(\go,\cdot))||_{L^2(\D)}\\
	&\le C_{\vartheta,\D}\frac{a_+(\go)}{a_-(\go)}||u_{N,\eps}(\go,\cdot)||_{H^{1+\gk_{a}}(\D)} h_\ell^{\min(\gk_{a},1)},
	\end{align*}
	where $C_{|\D|}$ is the Poincar\'e constant and $ C_{\vartheta,\D}>0$ is deterministic and only depends on the indicated parameters (see e.g. \cite[Chapter 8.3/8.5]{H10}).
	If $\cK$ is sample-independent (and thus $(h_\ell,\ell\in\N_0)$ is fixed for any $\go$), $a_+/a_-\in L^2(\gO;\R)$ and there exists a uniform bound $||u_{N,\eps}||_ {L^2(\gO;H^{1+\gk_a}(\D))}\le C_u$ in $N$ and $\eps$, we readily obtain
	\bee
	\|u_{N,\eps}-u_{N,\eps,\ell}\|_{L^2(\gO;V)}\le C_{\vartheta,\D}\E\left(\frac{a_+^2}{a_-^2}\right)^{1/2}C_u h_\ell^{\min(\gk_{a},1)}
	\eee
	by H\"older's inequality.
	We note that a uniform a-priori bound may require higher moments of $1/a_{N,\eps}$, $f$ and $g$ or even essential bounds on $a_{N,\eps}$ which are not ensured by Assumption~\ref{ass:EV}.
\end{ex}

\begin{rem}
	For jump-diffusion problems, we obtain, in general, a discretization error of order $\gk_{a}\in (1/2,1)$. We cannot expect the pathwise "full" order of convergence $\gk_{a}=1$ of the finite-dimensional discretization error, since the diffusion coefficient $a_{N,\eps}$ is almost surely discontinuous. 
	Most results that ensure  $H^2(\D)$-regularity of the pathwise weak solution $u_{N,\eps}(\go,\cdot)$ need that $a_{N,\eps}(\go,\cdot)$ is continuously differentiable or that $a_{N,\eps}(\go,\cdot)\in W^{1,\infty}(\D)$, see for instance \cite{E10}. 
	The latter would imply that $a_{N,\eps}(\go,\cdot)$ is continuous by the Sobolev embedding theorem, which contradicts our setting. 
	As we consider pathwise regularity, we can rely on results for deterministic elliptic problems with discontinuous coefficients.
	We refer to~\cite{P02} and the references therein, where the author emphasizes that the regularity of the solution depends on the shape and magnitude of the discontinuities. 
	For several examples $H^{1+\gk_a}$-regularity with $\gk_{a}<1$ is shown.
	In \cite{B70} the author states that without special treatment of the interfaces with respect to the triangulation one, in general, cannot expect a better pathwise convergence rate than $\gk_{a}=1/2$.
\end{rem}

\subsection{Adaptive triangulations}

In view of the previous remark, we aim to increase the order of convergence $\gk_{a}$ with respect to $h_\ell$.
For this, we employ path-dependent triangulations to match the interfaces generated by the samples of the jump-diffusion coefficient.
Hence, we need to reformulate the discrete problem with respect to $\go$, since the triangulation and matching basis functions may be sample-dependent. 
Given a fixed $\go$ and $\ell$, we consider a finite-dimensional subspace $\widehat V_\ell(\go)\subset V$ with sample-dependent basis $\{\widehat v_1(\go),\dots,\widehat v_{\widehat d_\ell}(\go)\}$ and stochastic dimension $\widehat d_\ell(\go)\in\N$. 
As before, we denote by $(\widehat h_\ell(\go),\ell\in\N_0)$ the sequence of (random) refinement parameters corresponding to the sequence of subspaces $(\widehat V_\ell(\go),\ell\in\N_0)$.
To be more precise, for a given random partition $\cT(\go)=(\cT_i,i=1\dots,\tau(\go))$ of $\D$, we choose a triangulation $\cK_{\ell}(\go)$ of $\D$ with respect to $\cT(\go)$ such that 
\bee
\cT(\go)\subset\cK_{\ell}(\go)\quad\text{and}\quad \widehat h_{\ell}(\go) :=\max_{K\in\cK_{\ell}(\go)}\text{diam}(K)\le \overline h_\ell\quad\text{for $\ell\in\N_0$,}
\eee
where $(\overline h_\ell,\ell\in\N_0)$ is a positive sequence of deterministic refinement thresholds, decreasing monotonically to zero. This guarantees that $\widehat h_\ell(\go)\to0$ almost surely, although the absolute speed of convergence may vary for each $\go$. Denoting by $\widehat\vartheta_\ell(\go)$ the minimal interior angle within $\cK_\ell(\go)$, we assume similarly to Example~\ref{ex:fem} that there exists a $\vartheta>0$ such that for $\bP$-almost all $\go$  
\be\label{eq:angle}
\inf_{\ell\in\N_0}\widehat \vartheta_\ell(\go)\ge\vartheta>0.
\ee
The  pathwise discrete variational problem in the sample-adaptive subspace $\widehat V_\ell(\go)$ now reads: 
Find $\widehat u_{N,\eps,\ell}(\go,\cdot)\in \widehat V_\ell(\go)$ such that 
\be\label{eq:fem_adaptive}
B_{a_{N,\eps}(\go)}(\widehat u_{N,\eps,\ell}(\go,\cdot),\widehat v_\ell(\go))=F_\go(\widehat v_\ell(\go))\quad\text{for all $\widehat v_\ell(\go)\in \widehat V_\ell(\go)\subset V$.}
\ee
Since the triangulation is aligned with the discontinuities of $a_{N,\eps}$, this approximation admits an increase of the (pathwise) order of convergence $\gk_{a}$ compared to the non-adaptive, deterministic Galerkin approximations. 
For piecewise linear basis functions $\widehat v_i$, 
we can expect
\be
\begin{split}\label{eq:afem1}
||u_{N,\eps}(\go,\cdot)-\widehat u_{N,\eps,\ell}(\go,\cdot)||_V
&\le C_{\vartheta,\D}\frac{a_+(\go)}{a_-(\go)}\Big(||u_{N,\eps}(\go,\cdot)||_V\\
&\qquad\qquad+\sum_{i=1}^{\tau(\go)}||\nabla\cdot\nabla u_{N,\eps}(\go,\cdot)||_{L^2(\cT_i(\go))}\Big) \widehat h_\ell(\go) \\
&=:C_{\vartheta,\D}\widehat C_u(\go) \widehat h_\ell(\go)
\end{split}
\ee
$\bP$-a.s. by results from domain decomposition methods for deterministic elliptic problems (e.g. \cite{B99,BMR05}).
We emphasize that with the assumption of Eq.~\eqref{eq:angle} the constant $C_{\vartheta,\D}>0$ is independent of $\go$.
This estimate holds provided the random partitions are polygonal and $u_{N,\eps}(\go,\cdot)$ is almost surely piecewise in $H^2(\cT_i)$. 
In this case, the adaptive triangulations perfectly fit the random subdomains in $\D$ and the order of convergence is the same as if $u_{N,\eps}(\go,\cdot)\in H^2(\D)$.
If the random partitions are not polygonal but form $C^2$-interfaces within $\D$, we obtain log-linear rates:
\begin{align}\label{eq:afem2}
||u_{N,\eps}(\go,\cdot)-\widehat u_{N,\eps,\ell}(\go,\cdot)||_V
\le  C_{\vartheta,\D}\widehat C_u(\go) \widehat h_\ell(\go)|\log(\widehat h_\ell(\go))|^{1/2} 
\end{align}
with $\widehat C_u(\go)$ as in Ineq.~\eqref{eq:afem1}, see \cite{CZ98}. 

\begin{rem}
	Based on Ass.~\ref{ass:fd} and the pathwise estimates in Eqs.~\eqref{eq:afem1}, \eqref{eq:afem2}, the question arises whether it is possible to derive estimates in a mean-square sense 
	\bee
	\|u_{N,\eps}-\widehat u_{N,\eps,\ell}\|_{L^2(\gO;V)}\le \widetilde C_u \E(\widehat h_\ell^{2})^{1/2},
	\eee
	where the constant $\widetilde C_u$ is independent of $N$ and $\eps$. 	
	As the independence of $\eps$ is not an issue, a uniform estimate with respect to $N$ requires further summability conditions (i.e. $\gb\ge2\alpha$) on the eigenvalues in Ass.~\ref{ass:EV} and would result in a piecewise differentiable diffusion coefficient $a$. 
	To illustrate this, we consider the identity 
	\bee
	-\nabla\cdot(a_{N,\eps}(\go,\cdot)\nabla u_{N,\eps}(\go,\cdot))=-\nabla a_{N,\eps}(\go,\cdot)\cdot\nabla u_{N,\eps}(\go,\cdot) -a_{N,\eps}(\go,\cdot) \nabla\cdot\nabla u_{N,\eps}(\go,\cdot)=f(\go,\cdot)
	\eee
	on some partition element $\cT_i(\go)$. Rearranging terms and taking expectations then yields by H\"older's inequality
	\bee
	\E\big(\sum_{i=1}^{\tau(\go)}||\nabla\cdot\nabla u_{N,\eps}(\go,\cdot)||_{L^2(\cT_i(\go))}\big)\le\E(\tau)^{1/2}\E\Big(\frac{2||f||^2_H+2||\nabla (\ol a +\Phi(W_N))||_{L^\infty(\D)}^2||\nabla u_{N,\eps}||^2_{H}}{a_{N,\eps,-}^2}\Big).
	\eee
	If $f$ and $1/a_-$ have sufficiently high moments and the eigenvalues $\eta_i$ decay fast enough, the right hand side may be bounded uniformly with respect to $N$, then it is straightforward to derive the corresponding $L^2(\gO;V)$-estimates. These assumptions, however, exclude the important cases where the covariance operator of $W$ is of Brownian-motion-type or exponential.
	Practically, as we show in Section~\ref{sec:num}, the pathwise adaptive convergence rates may also be recovered for the $L^2(\gO,V)$-error if $a$ has only piecewise H\"older-continuous trajectories. 
\end{rem}
In view of the preceding remarks and Example~\ref{ex:fem}, we make the following assumption on the mean-square discretization error: 
\begin{assumption}\label{ass:fd}
	There exist constants $C_{u,a},C_{u,n},\gk_a,\gk_n>0$, such that for any $N,\ell\in\N_0$ and $\eps>0$, the finite-dimensional approximation errors of $u_{N,\eps}$ in the subspaces $\widehat V_\ell$ resp. $V_\ell$ are bounded by  
	\bee
	\|u_{N,\eps}-\widehat u_{N,\eps,\ell}\|_{L^2(\gO;V)}\le C_{u,a}\E(\widehat h_\ell^{2\gk_a})^{1/2} \;\text{ resp. }
	\eee
	\bee
	\|u_{N,\eps}-u_{N,\eps,\ell}\|_{L^2(\gO;V)}\le C_{u,n} h_\ell^{\gk_n}.
	\eee
	The constants $C_{u,a},C_{u,n}$ may depend on $a,f$ and $g$, but are independent of $\widehat h_\ell,h_\ell,\gk_a$ and $\gk_n$.
\end{assumption}
Note that in general $1\ge\gk_a>\gk_n>0$ by the previous observations. We consider $\E(\widehat h_\ell^{2\gk_a})^{1/2}$ instead of $\ol h_\ell^{\gk_a}$ as both quantities  depend to a great extend on the geometry introduced by $a$. The parameter $\ol h_\ell$ has been merely introduced to ensure  that $\lim_{\ell\to\infty}\E(\widehat h_\ell^2)=0$.

\section{Estimation of expectations by Monte Carlo methods}\label{sec:mlmc}
As we are able to generate samples from $u_{N,\eps}$ and control for the discretization error in each sample,  
we may estimate the expected value $\E(u)$ of the weak solution to Eq.~\eqref{eq:elliptic}.
We focus on \MLMC\, estimators as they are easily implemented and do not require much regularity of $u$.
Monte Carlo estimators introduce an additional statistical bias besides the error contributions of $a_{N,\eps}$ and the pathwise discretization error from Section~\ref{sec:fem}. 
However, this error can be controlled under natural assumptions and it may be equilibrated according to the other error terms.   
In this section, we first recall briefly standard \MC\, and \MLMC\, methods to estimate $\E(u)$ and then control for the mean-squared error in both algorithms.
We also suggest a modification of the \MLMC\, estimation to increase computational efficiency before we verify our results on several numerical examples in Section~\ref{sec:num}.

\subsection{\MC\, and \MLMC\, estimators}
Consider a sequence $(u^{(i)}, i\in\N)$ of i.i.d. copies of the $V$-valued random variable $u$.
For $M\in\N$ independent samples, the \textit{\MC\,estimator} of $\E(u)$ is defined as
\bee
E_M(u):=\frac{1}{M}\sum_{i=1}^Mu^{(i)}.
\eee
Since we are only able to draw samples of the approximated discrete solution $\widehat u_{N,\eps,\ell}$ as introduced in Section~\ref{sec:fem}, we aim to control the \textit{root mean-squared error} (RMSE)
% \bee
$||\E(u)-E_M(\widehat u_{N,\eps,\ell})||_{L^2(\gO;V)}$.
% \eee
For notational convenience, we focus only on the adaptive discretization $\widehat u_{N,\eps,\ell}$ with mean-square refinement $\E(\widehat h_\ell^2)^{1/2}$ and converge rate $\gk_a$ in this section. However, all results also hold in the non-adaptive case where we replace $\E(\widehat h_\ell^2)^{1/2}$ by $h_\ell$ and $\gk_a$ by $\gk_n$.
To derive a bound for the RMSE, we need the following standard result.
\begin{lem}\label{lem:mc}
	Let $M\in\N$ and $u\in L^2(\gO;V)$. Then
	\bee
	||\E(u)-E_M(u)||_{L^2(\gO;V)}\le\frac{||u||_{L^2(\gO;V)}}{\sqrt M}\quad\text{and}\quad||E_M(u)||_{L^2(\gO;V)}\le||u||_{L^2(\gO;V)}.
	\eee
\end{lem}
\begin{thm}\label{thm:mc}
	Let Assumptions~\ref{ass:EV} and ~\ref{ass:fd} hold such that $t:=(2/p+1/q+1/s)^{-1}\ge2$, where $p\in[1, (\eta^*)^{-1})$ for $\eta^*$ as in Lemma~\ref{lem:a}.
	Then, for any $M,N\in\N$ and $\eps>0$, 
	\bee
	||\E(u)-E_M(\widehat u_{N,\eps,\ell})||_{L^2(\gO;V)}\le C_{MC}\left(\frac{1}{\sqrt M}+\Xi_N^{1/2}+\eps^{1/s}+\E(\widehat h_\ell^{2\gk_a})^{1/2}\right),
	\eee
	where $C_{MC}>0$ is a constant independent of $M,\Xi_N,\eps$ and $\widehat h_\ell$. 
\end{thm}
\begin{proof}
	We apply the triangle inequality to obtain
	\begin{align*}  
	||\E(u)-E_M(\widehat u_{N,\eps,\ell})||_{L^2(\gO;V)}&\le ||\E(u)-E_M(u)||_{L^2(\gO;V)} +||E_M(u-u_{N,\eps})||_{L^2(\gO;V)}\\&\quad+ ||E_M(u_{N,\eps}-\widehat u_{N,\eps,\ell})||_{L^2(\gO;V)}.
	\end{align*}
	By the first part of Lemma~\ref{lem:mc} and Theorem~\ref{thm:pw_u}, we bound the first term by
	\begin{align*}
	||\E(u)-E_M(u)||_{L^2(\gO;V)}\le C(a_-,\D,p)(||f||_{L^q(\gO;H)}+||g||_{L^q(\gO;\gG_2)})\frac{1}{\sqrt M}.
	\end{align*}
	The second part of Lemma~\ref{lem:mc} yields with the estimate of Theorem~\ref{thm:u_error} 
	\begin{align*}
	||E_M(u-u_{N,\eps})||_{L^2(\gO;V)}&\le\widetilde C(a, f,g,\D)||a^{-1}_{N,\eps}||_{L^p(\gO;\R)}||a-a_{N,\eps}||_{L^s(\gO;L^\infty(\D))}\\
	&\le C(a,f,g,\D)(\Xi_N^{1/2}+\eps^{1/s})
	\end{align*}
	where $\widetilde C(a,f,g,\D), C(a,f,g,\D)>0$ are independent of $N$ and $\eps$, because $||1/a_{N,\eps,-}||_{L^p(\gO;R)}$ is bounded uniformly with respect to these parameters by Lemma~\ref{lem:a_n}.
	Finally, Assumption~\ref{ass:fd} and Lemma~\ref{lem:mc} yield for the third term 
	\bee
	||E_M(u_{N,\eps}-\widehat u_{N,\eps,\ell})||_{L^2(\gO;V)}\le C_{u,a} \E(\widehat h_\ell^{2\gk_a})^{1/2}
	\eee
	where $C_{u,a}$ is also independent of $N$ and $\eps$ by assumption.
	
\end{proof}
\begin{rem}\label{rem:mc}
	The estimate in Theorem~\ref{thm:mc} suggests that all four error contributions should be equilibrated to obtain a RMSE of order $\E(\widehat h_\ell^{2\gk_a})^{1/2}$.
	For this, we may choose $M,N$ and $\eps$ such that
	\bee
	M^{-1}\simeq \Xi_N \simeq\eps^{2/s}\simeq \E(\widehat h_\ell^{2\gk_a}).
	\eee  
	While this is straightforward for $M$ and usually also for $\eps$, the choice of the truncation index $N$ involves a few difficulties. 
	In general, the eigenvalues $(\eta_i, i\in\N)$ of the covariance operator $Q:H\to H$ will not be available in closed form.
	The decay parameter $\gb>0$ may be unknown and only be bounded from below, which would result in an overestimation of the term $\Xi_N$.
	One possibility to find $N$ is applicable if $Q$ is an integral operator of the form
	\bee
	(Q\varphi)[x]=v\int_\D k_Q(x,y)\varphi(y)dy,\quad \varphi\in H,
	\eee
	with some nonnegative, symmetric and bounded kernel function $k_Q:\D^2\to\R$ and $v>0$.
	In this case, the eigenvalues of $Q$ fulfill the identity $v\int_Ddx=\sum_{i\in\N}\eta_i$ (see \cite{E09, LZ04}).
	Operators with this property are widely used in practice and include for instance Mat\'ern class, $\gg$-exponential and rational quadratic covariance functions (see \cite{RW06}).
	The first eigenvalues of $Q$ have to be, in any case, determined (numerically) to approximate the Gaussian field $W$, so we select $N$ such that 
	\bee
	\Xi_N=v\int_\D dx-\sum_{i=1}^N\eta_i\stackrel{!}{\simeq} \E(\widehat h_\ell^{2\gk_a}).
	\eee
\end{rem}

In most cases, the sampling of $a_{N,\eps}$ and $u_{N,\eps, \ell}$ for given boundary data will be computationally expensive:
If the eigenvalues of $Q$ decay slowly, it is necessary to include a large number of terms in the \KL expansion to achieve $\Xi_N\simeq \E(\widehat h_\ell^{2\gk_a})$.
In addition, sampling of the sequence $(\widetilde P_i, i\in\N)$ might also be time-consuming if a small error $\E(|\widetilde P_i-P_i|^s)\le\eps$ is desired.
Given a sample of $a_{N,\eps}$, one has then to rely on numerical integration schemes to calculate the entries of the stiffness matrix $\mathbf A(\go)$ and the load vector $\mathbf F(\go)$, and solve a possibly large system of linear equations.
This motivates the use of advanced Monte Carlo techniques, such as \MLMC, to achieve essentially the same accuracy with reduced computational effort. 
We briefly recall the idea of the \MLMC\, sampling in the following. 

For $L\in\N$ we consider finite-dimensional subspaces $\widehat V_0\subset\dots\subset \widehat V_L$ of $V$ with refinement sizes $\widehat h_0>\dots>\widehat h_L>0$ and approximation parameters $N_0<\dots<N_L$ and $\eps_0>\dots>\eps_\ell$. 
We define $u_{N_{-1},\eps_{-1},-1}:=0$ and expand the ``finest level approximation'' $\widehat u_{N_L,\eps_L,L}$ into a telescopic sum to obtain 
\bee
\E(\widehat u_{N_L,\eps_L,L})=\sum_{\ell=0}^L\E(\widehat u_{N_\ell,\eps_\ell,\ell}-\widehat u_{N_{\ell-1},\eps_{\ell-1},{\ell-1}}).
\eee
Instead of estimating the left hand side by the ordinary Monte Carlo method, we estimate the expected corrections $\E(\widehat u_{N_\ell,\eps_\ell,\ell}-\widehat u_{N_{\ell-1},\eps_{\ell-1},{\ell-1}})$
by generating $M_\ell$ independent realizations $\widehat u^{(i,\ell)}_{N_\ell,\eps_\ell,\ell}-\widehat u^{(i,\ell)}_{N_{\ell-1},\eps_{\ell-1},{\ell-1}}$ on each level and calculating the Monte Carlo estimator $E_{M_\ell}(\widehat u_{N_\ell,\eps_\ell,\ell}-\widehat u_{N_{\ell-1},\eps_{\ell-1},{\ell-1}})$.
The \textit{\MLMC\ estimator} of $u_{N_L,\eps_L,L}$ is then defined as
\begin{equation}\label{eq:mlmc}
E^L(\widehat u_{N_L,\eps_L,L}):=\sum_{\ell=0}^LE_{M_\ell}(\widehat u_{N_\ell,\eps_\ell,\ell}-\widehat u_{N_{\ell-1},\eps_{\ell-1},{\ell-1}})=\sum_{\ell=0}^L\frac{1}{M_\ell}\sum_{i=1}^{M_\ell}\widehat u^{(i,\ell)}_{N_\ell,\eps_\ell,\ell}-\widehat u^{(i,\ell)}_{N_{\ell-1},\eps_{\ell-1},{\ell-1}}
\end{equation}
To achieve a desired target RMSE of $\epsilon_{RMSE}>0$, this estimator requires less computational effort than the standard \MC\, approach under certain assumptions.
This, by now, classical result was proven in \cite[Theorem 3.1]{G08}.
The proof is rather general and may readily be transferred to the problem of estimating moments of random PDEs (see~\cite{BSZ11}).
\begin{thm}\label{thm:mlmc}
	Let Assumptions~\ref{ass:EV} and ~\ref{ass:fd} hold such that $t:=(2/p+1/q+1/s)^{-1}\ge2$, where $p\in[1, (\eta^*)^{-1})$ for $\eta^*$ as in Lemma~\ref{lem:a}.
	For $L\in\N$, let $\widehat h_\ell>0$, $M_\ell,N_\ell\in\N$ and $\eps_\ell>0$ be the level dependent approximation parameters for any $\ell=0,\dots,L$ such that $\widehat h_\ell, \eps_\ell$ are decreasing and $N_\ell$ is increasing with respect to $\ell$.
	Then the \MLMC\ estimator admits the bound
	\begin{align*}
	||\E(u)-E^L(\widehat u_{N_L,\eps_L,L})||_{L^2(\gO;V)}\le &C\Big(\Xi_{N_L}^{1/2}+\eps_L^{1/s}+\E(\widehat h_L^{2\gk_a})^{1/2}+\frac{1}{\sqrt{M_0}}\\
	&\quad+\sum_{\ell=0}^{L-1}\frac{\Xi_{N_\ell}^{1/2}+\eps^{1/s}_\ell+\E(\widehat h_\ell^{2\gk_a})^{1/2}}{\sqrt{ M_{\ell+1}}}\Big),
	\end{align*}
	where $C>0$ is a constant independent of $L$ and the level-dependent approximation parameters.
\end{thm}
\begin{proof}
	Using the triangle inequality and Jensen's inequality for expectations, we observe that
	\begin{align*}
	||\E(u)&-E^L(\widehat u_{N_L,\eps_L,L})||_{L^2(\gO;V)}\\
	\le&||\E(u)-\E(\widehat u_{N_L,\eps_L,L})||_{L^2(\gO;V)}+||\E(\widehat u_{N_L,\eps_L,L})-E^L(\widehat u_{N_L,\eps_L,L})||_{L^2(\gO;V)}\\
	\le&\underbrace{||u-u_{N_L,\eps_L}||_{L^2(\gO;V)}+||u_{N_L,\eps_L}-\widehat u_{N_L,\eps_L,L}||_{L^2(\gO;V)}}_{:=I}+\underbrace{||\E(\widehat u_{N_L,\eps_L,L})-E^L(\widehat u_{N_L,\eps_L,L})||_{L^2(\gO;V)}}_{:=II}.
	\end{align*}
	Theorem~\ref{thm:u_error} and Assumption~\ref{ass:fd} give a bound for the first term by
	\bee
	I\le C(a,f,g,\D)(\Xi_{N_L}^{1/2}+\eps_L^{1/s})+C_{u,a}\E(\widehat h_L^{2\gk_a})^{1/2},
	\eee 
	where $C(a,f,g,\D)>0$ is an independent constant.
	For the second error term, the definition of $E^L$ in Eq.~\eqref{eq:mlmc} together with Lemma~\ref{lem:mc} yield
	\begin{align*}
	II&\le\sum_{\ell=0}^L||\E(\widehat u_{N_\ell,\eps_\ell,\ell}-\widehat u_{N_{\ell-1},\eps_{\ell-1},\ell-1})-E_{M_\ell}( \widehat u_{N_\ell,\eps_\ell,\ell}-\widehat u_{N_{\ell-1},\eps_{\ell-1},\ell-1})||_{L^2(\gO;V)}\\
	&\le\sum_{\ell=0}^L\frac{1}{\sqrt{M_\ell}}||\widehat u_{N_\ell,\eps_\ell,\ell}-\widehat u_{N_{\ell-1},\eps_{\ell-1},\ell-1}||_{L^2(\gO;V)}\\
	&\le\sum_{\ell=0}^L\Big(\frac{1}{\sqrt{M_\ell}}||\widehat u_{N_\ell,\eps_\ell,\ell}-u||_{L^2(\gO;V)}+\frac{1}{\sqrt{M_\ell}}||u-\widehat u_{N_{\ell-1},\eps_{\ell-1},\ell-1}||_{L^2(\gO;V)}\Big).
	\end{align*}
	Now, the terms $||\widehat u_{N_\ell,\eps_\ell,\ell}-u||_{L^2(\gO;V)}$ may be treated analogously to $I$:
	\bee
	||\widehat u_{N_\ell,\eps_\ell,\ell}-u||_{L^2(\gO;V)}\le C(a,f,g,\D)(\Xi_{N_\ell}^{1/2}+\eps_\ell^{1/s})+C_{u,a}\E(\widehat h_\ell^{2\gk_a})^{1/2}.
	\eee
	The remaining term $||u-\widehat u_{N_{-1},\eps_{-1},-1}||_{L^2(\gO;V)}$ is bounded by Theorem~\ref{thm:pw_u} via
	\bee
	||u-\widehat u_{N_{-1},\eps_{-1},-1}||_{L^2(\gO;V)}=||u||_{L^2(\gO;V)}\le C(a_-,\D,p)(||f||_{L^q(\gO;H)}+||g||_{L^q(\gO;\gG_2)}).
	\eee
	Finally, we arrive at the estimate
	\begin{align*}
	I+II&\le   C(a,f,g,\D)\Big(\Xi_{N_L}^{1/2}+\eps_L^{1/s}+\sum_{\ell=1}^L\frac{\Xi_{N_\ell}^{1/2}+\eps_\ell^{1/s}+\Xi_{N_\ell-1}^{1/2}+\eps_{\ell-1}^{1/s}}{\sqrt{M_\ell}}\Big)\\
	&\quad+C_{u,a}\Big(\E(\widehat h_L^{2\gk_a})^{1/2}+\sum_{\ell=1}^L\frac{\E(\widehat h_\ell^{2\gk_a})^{1/2}+\E(\widehat h_{\ell-1}^{2\gk_a})^{1/2}}{\sqrt{M_\ell}}\Big)+\frac{C(a_-,\D,p)(||f||_{L^q(\gO;H)}+||g||_{L^q(\gO;\gG_2)})}{\sqrt{M_0}}\\
	&\le C\Big(\Xi_{N_L}^{1/2}+\eps_L^{1/s}+\E(\widehat h_L^{2\gk_a})^{1/2}+
	\sum_{\ell=0}^{L-1}\frac{\Xi_{N_\ell}^{1/2}+\eps_\ell^{1/s}+\E(\widehat h_\ell^{2\gk_a})^{1/2}}{\sqrt{M_{\ell+1}}}+\frac{1}{\sqrt{M_0}}\Big).
	\end{align*}
	For the last inequality, we have used that $\Xi_{N_\ell}\le \Xi_{N_{\ell-1}}, \eps_\ell\le\eps_{\ell-1}$ and $\E(\widehat h_\ell^{2\gk_a})^{1/2}\le \E(\widehat h_{\ell-1}^{2\gk_a})^{1/2}$ with
	\bee
	C:=\max(2C(a,f,g,\D),2C_{u,a},C(a_-,\D,p)(||f||_{L^q(\gO;H)}+||g||_{L^q(\gO;\gG_2)})).
	\eee 
\end{proof}

Regarding a single realization of $\widehat u_{N_\ell,\eps_\ell,\ell}$, we want the pathwise error $||u(\go,\cdot)-\widehat u_{N_\ell,\eps_\ell,\ell}(\go,\cdot)||_V$ to decrease as $\ell$ increases,
so naturally the parameters $\eps_\ell$ and $\E(\widehat h_\ell^{2\gk_a})^{1/2}$ should decrease in $\ell$ and $N_\ell$ increase in $\ell$.
For example, by using a sequence of refining grids, the refinement parameter $\E(\widehat h_\ell^{2})^{1/2}$ may be divided roughly by a factor of two in each level, i.e. $2\E(\widehat h_\ell^{2})^{1/2}\approx\E(\widehat h_{\ell-1}^{2})^{1/2}$ for any $\ell\in\N$.
Similar refining factors may be imposed for the sum of the remaining eigenvalues $\Xi_{N_\ell}$ and the sampling errors $\eps_\ell$.
One advantage of the \MLMC\, estimator is that we are now able to even out the error contributions of the sampling bias $||\widehat u_{N_\ell,\eps_\ell,\ell}-u||_V$ and the statistical error (with respect to $M_\ell$) on each level. 
This is achieved by generating relatively few of the accurate, but expensive, samples for large $\ell$ and generating more of the cheap, but less accurate, samples on the lower levels.

\begin{cor}\label{cor:mlmc}
	Let the assumptions of Theorem~\ref{thm:mlmc} hold. For $L\in\N$ and given refinement parameters $\E(\widehat h_0^{2\gk_a})^{1/2}>\dots>\E(\widehat h_L^{2\gk_a})^{1/2}>0$ choose $N_\ell\in\N$ and $\eps_\ell>0$ such that
	\bee
	\Xi_{N_\ell}\simeq\eps^{2/s}_\ell\simeq \E(\widehat h_\ell^{2\gk_a})\quad\text{for $\ell=0,\dots,L$},
	\eee
	and $M_\ell\in\N$ such that for some arbitrary $\nu>0$
	\bee
	M_0^{-1}\simeq \E(\widehat h_L^{2\gk_a})\quad\text{and}\quad M_\ell^{-1}\simeq \E(\widehat h_L^{2\gk_a})^{-1}\E(\widehat h_{\ell-1}^{2\gk_a})(\ell+1)^{-2(1+\nu)}\quad\text{for $\ell=1,\dots,L$}.
	\eee
	The RMSE of the corresponding \MLMC\, estimator is then of order $\E(\widehat h_L^{2\gk_a})^{1/2}$:
	\begin{align*}
	||\E(u)-E^L(\widehat u_{N_L,\eps_L,L})||_{L^2(\gO;V)}= \cO (\E(\widehat h_L^{2\gk_a})^{1/2}).
	\end{align*}
\end{cor}
\begin{proof}
	With the above choice of the approximation parameters $N_\ell, \eps_\ell$ and Theorem~\ref{thm:mlmc}, we obtain
	\begin{align*}
	||\E(u)-E^L(\widehat u_{N_L,\eps_L,L})&||_{L^2(\gO;V)} \le C\Big(4\E(\widehat h_L^{2\gk_a})^{1/2}+\sum_{\ell=0}^{L-1} \frac{\E(\widehat h_\ell^{2\gk_a})^{1/2}}{\sqrt{M_{\ell+1}}}\Big)\\
	&\le C\Big(4\E(\widehat h_L^{2\gk_a})^{1/2}+\E(\widehat h_0^{2\gk_a})^{1/2}\E(\widehat h_L^{2\gk_a})^{1/2}
	+\sum_{\ell=1}^L \E(\widehat h_L^{2\gk_a})^{1/2} (\ell+1)^{-1-\nu}\Big) \\
	&\le C\big(3+\E(\widehat h_0^{2\gk_a})^{1/2}+\zeta(1+\nu)\big)\E(\widehat h_L^{2\gk_a})^{1/2},
	\end{align*}
	where $\zeta(\cdot)$ is the Riemann zeta function.
\end{proof}
For the level-dependent choice of $N_\ell$ and $\eps_\ell$ we refer to Remark~\ref{rem:mc}.
In the remainder of this section, we introduce a modification of the \MLMC\, method to further reduce computational complexity.  

\subsection{Bootstrapping \MLMC}
Recall the \MLMC\, estimator
\bee
E^L(\widehat u_{N_L,\eps_L,L})=\sum_{\ell=0}^L\frac{1}{M_\ell}\sum_{i=1}^{M_\ell} \Big(\widehat  u^{(i,\ell)}_{N_\ell,\eps_\ell,\ell}-\widehat u^{(i,\ell)}_{N_{\ell-1},\eps_{\ell-1},{\ell-1}}\Big)
\eee
of  $\E(u)$ as in Eq.~\eqref{eq:mlmc}, where the terms in the second sum are independent copies of the corrections $\widehat u_{N_\ell,\eps_\ell,\ell}-\widehat u_{N_{\ell-1},\eps_{\ell-1},\ell-1}$. In total, one has to generate $M_\ell+M_{\ell+1}$ samples of $\widehat u_{N_\ell,\eps_\ell,\ell}$ for each $\ell=0,\dots,L$ (where we have set $M_{L+1}:=0$). This could be expensive even in low dimensions $d$.
We can reduce this effort if we ``recycle'' the already available samples and generate the differences  
\bee
\widehat u^{(i,\ell)}_{N_\ell,\eps_\ell,\ell}-\widehat u^{(i,\ell)}_{N_{\ell-1},\eps_{\ell-1},{\ell-1}}
\quad\text{and}\quad
\widehat u^{(i,\ell)}_{N_{\ell+1},\eps_{\ell+1},{\ell+1}}-\widehat u^{(i,\ell)}_{N_\ell,\eps_\ell,\ell}
\eee
based on the same realization $\widehat u^{(i,\ell)}_{N_\ell,\eps_\ell,\ell}$ for $\ell=0,\dots,L$.
That is, we drop the second superscript $\ell$ in $\widehat u^{(i,\ell)}_{N_\ell,\eps_\ell,\ell}$ and arrive at the \textit{bootstrap \MLMC\, estimator} 
\bee
E^L_{BS}(\widehat u_{N_L,\eps_L,L}):=\sum_{\ell=0}^L\frac{1}{M_\ell}\sum_{i=1}^{M_\ell}\widehat u^{(i)}_{N_\ell,\eps_\ell,\ell}-\widehat u^{(i)}_{N_{\ell-1},\eps_{\ell-1},{\ell-1}}.
\eee
This entails generating $M_\ell$ realizations of the random variable $\widehat u_{N_\ell,\eps_\ell,\ell}$ instead of $M_\ell+M_{\ell+1}$.
The samples $u^{(i)}_{N_\ell,\eps_\ell,\ell}$ are then independent in $i$, but not anymore across all levels $\ell$ for a fixed index $i$.
The bootstrap estimator is unbiased, i.e $\E(E^L_{BS}(\widehat u_{N_L,\eps_L,L}))=\E(\widehat u_{N_L,\eps_L,L})$,
and it holds
\bee
\lim_{L\to+\infty}\E(E^L_{BS}(\widehat u_{N_L,\eps_L,L}))=\lim_{L\to+\infty}\E(E^L(\widehat u_{N_L,\eps_L,L}))=\lim_{L\to+\infty}\E(u_{N_L,\eps_L}) = \E(u).
\eee
The introduced modification is a simplified version of the \textit{Multifidelity Monte Carlo estimator} (see~\cite{GPW16}).
Under suitable assumptions on the variance of $u$, it is shown in~\cite{GPW16} that the Multifidelity Monte Carlo approach achieves the same rate of convergence as the standard \MLMC\, method
with reduced computational effort. 
The bootstrap estimator corresponds to a Multifidelity Monte Carlo estimator where the weighting coefficients for all level corrections $\widehat u_{N_\ell,\eps_\ell,\ell}-\widehat u_{N_{\ell-1},\eps_{\ell-1},{\ell-1}}$ are set equal to one. 

We emphasize that the error bounds derived in Thm.~\ref{thm:mlmc} and Cor.~\ref{cor:mlmc} do not require independence of the sampled differences $\widehat u_{N_\ell,\eps_\ell,\ell}-\widehat u_{N_{\ell-1},\eps_{\ell-1},{\ell-1}}$ across the levels $\ell$.
Thus, the asymptotic order of convergence also holds for the bootstrapping estimator. 
However, we have now introduced additional correlation across the levels, which may entail higher variance, and thus slower convergence of the bootstrapping method.
The variance of the standard \MLMC\, estimator is easily calculated as
\bee
\var(E^L(\widehat u_{N_L,\eps_L,L}))=\sum_{\ell=0}^L\frac{\var(\widehat u_{N_\ell,\eps_\ell,\ell}-\widehat u_{N_{\ell-1},\eps_{\ell-1},{\ell-1}})}{M_\ell}, 
\eee
whereas 
\begin{align*}
\var(E_{BS}^L(\widehat u_{N_L,\eps_L,L}))&=\var\Big(\sum_{\ell=0}^L\frac{1}{M_\ell}\sum_{i=1}^{M_\ell}\widehat u^{(i)}_{N_\ell,\eps_\ell,\ell}-\widehat u^{(i)}_{N_{\ell-1},\eps_{\ell-1},{\ell-1}}\Big)\\
&=\var\Big(\sum_{i=1}^{M_0-M_1}\frac{\widehat u^{(i)}_{N_0,\eps_0,0}}{M_0}+\sum_{i=M_0-M_1+1}^{M_0-M_2}\frac{\widehat u^{(i)}_{N_1,\eps_1,1}-\widehat u^{(i)}_{N_0,\eps_0,0}}{M_1}
+\frac{\widehat u^{(i)}_{N_0,\eps_0,0}}{M_0}+\dots\\
&\qquad+\sum_{i=M_0-M_L+1}^{M_0}\sum_{\ell=0}^L\frac{\widehat u^{(i)}_{N_{\ell},\eps_\ell,\ell}-\widehat u^{(i)}_{N_{\ell-1},\eps_{\ell-1},{\ell-1}}}{M_\ell}\Big)\\
&=\sum_{\ell=0}^L(M_{\ell}-M_{\ell+1})\var\left(\sum_{k=0}^{\ell}\frac{\widehat u_{N_k,\eps_k,k}-\widehat u_{N_{k-1},\eps_{k-1},{k-1}}}{M_k}\right)\\
\end{align*}

In case that the differences $\widehat u_{N_\ell,\eps_\ell,\ell}-\widehat u_{N_{\ell-1},\eps_{\ell-1},{\ell-1}}$ are positively correlated across the levels, 
we trade in simulation time for a possibly higher RMSE, where the ratio of this trade is problem-dependent and hard to access beforehand. 
Nevertheless, we will compare this modified estimator with the standard \MLMC\, estimator in different scenarios to show its advantages. 

\section{Numerical examples}\label{sec:num}
Throughout this section, we plot the error rates against the smallest estimated root-mean-square refinement size $\E(\widehat h_L^2)^{1/2}$.
For the standard, non-adaptive algorithms this corresponds to the preset deterministic refinement size $h_L$ and all error contributions may be equilibrated a-priori as in Cor.~\ref{cor:mlmc}.
In the adaptive algorithm, to align the discretization error with the error contributions of $\sum_{i>N}\eta_i$, $\eps$ and the statistical error, we sample a few realizations of the diffusion coefficient before the start of the Monte Carlo loop.
This allows us to estimate the values of $\E(\widehat h_\ell^2)^{1/2}$ for each $\ell=0,\dots,L$ and choose $N$, $\eps$ and the number of samples on each level accordingly.
All numerical examples are implemented with MATLAB and calculated on a workstation with 16 GB Memory and Intel quadcore processor with 3.4 GHz.

\subsection{Numerical examples in 1D}
For all test scenarios in this subsection, we consider the diffusion problem~\eqref{eq:elliptic} in the one dimensional domain $\D=(0,1)$ with homogeneous Dirichlet boundary conditions, i.e. $\gG_1=\partial\D$, and source term $f\equiv1$. 
The deterministic part of the diffusion coefficient is $\ol a\equiv 0$ and we consider a log-Gaussian component, i.e. $\Phi(w)=\exp(w)$, where the Gaussian field $W$ is characterized by either the \textit{Brownian motion covariance operator} $Q_{BM}:H\to H$ with 
\bee
[Q_{BM}\varphi](y):=\int_\D min(x,y)\varphi(x)dx,\quad \varphi\in H
\eee
or the \textit{squared exponential covariance operator}
\bee
[Q_{SE}\varphi](y):=v\int_\D \exp\left(\frac{-|x-y|^2}{2r^2}\right)\varphi(x)dx,\quad \varphi\in H
\eee
with variance parameter $v>0$ and correlation length $r>0$. 
The eigenbasis of $Q_{BM}$ is given by
\bee
\eta_i=\Big(\frac{8}{(2i+1)\pi}\Big)^2,\quad e_i(x)=\sin\Big(\frac{(2i+1)\pi x}{2}\Big),\quad i\in\N_0,
\eee
see for example \cite[p. 46 ff.]{A09},
where the spectral basis of $Q_{SE}$ may be efficiently approximated by Nystr\"om's method, see \cite{RW06}.
The number of partition elements is given by $\tau=\cP+2$, where $\cP$ is a Poisson-distributed random variable with intensity parameter $10$.  
On average, this splits the domain in $12$ disjoint intervals and the diffusion coefficient has almost surely at least one discontinuity. 
The random positions of the $\tau-1$ jumps  $x_1,\dots,x_{\tau-1}\subset\D$ in the interior of $\D$ are uniformly distributed over $\D$, generating the random partition $\cT=\{(0,x_1],(x_1,x_2],\dots,(x_{\tau-1},1)\}$ for each realization of $\tau$ and the jump positions $x_1,\dots,x_{\tau-1}$.
This fits into our framework of the jump-diffusion coefficient by setting $\gl=12 \gL$, where $\gL$ denotes the Lebesgue measure on $((0,1),\cB(0,1))$:
The uniform distribution of the discontinuities on $\D$ corresponds to a distribution with respect to $\gL$ on $\D$ and the multiplication with $12$ ensures that $\E(\tau)$ is as desired. 
In the subsequent examples we vary the distribution of the jump heights $P_i$.

To obtain pathwise approximations of the samples $u_{N,\eps}(\go,\cdot)$, we use non-adaptive and adaptive piecewise linear elements and compare both approaches. 
In addition, we combine each discretization method with regular and bootstrapping \MLMC\ sampling, so in total we compare four different algorithms. 
The adaptive triangulation is based on each sampled partition $\cT(\go)$ as described in Section~\ref{sec:fem}, see Fig.~\ref{fig:BB_Unif} and \ref{fig:BB_GIG}.
%Within one partition element $\cT_K$ of $\cT$, we generate equally big simplices $K$ such that $diam(K)\le h_\ell$ is fulfilled with a minimum number of simplices in each $K\in\cT_K$. 
In the graphs below, we plot the RMSE of the adaptive algorithms against the inverse estimated refinement size $\E(\widehat h_\ell^2)^{-1/2}$, for the non-adaptive algorithms this corresponds to the (deterministic) parameters $h_\ell^{-1}$. 
The entries of the stiffness matrix are approximated by the midpoint rule, which ensures a pathwise error of order $h_\ell^3$ on each simplex $K$. To ensure that this is sufficiently precise, we repeated our experiments with a five-point Gauss-Legendre quadrature, which did not entail significant changes.
In the \MLMC\, algorithm, the non-adaptive triangulations are generated with refinement $h_\ell=2^{-\ell-1}$, whereas we set the same threshold as maximum refinement size $\overline h_\ell=h_\ell=2^{-\ell-1}$ in the adaptive algorithm. 
As realized and maximal values of $\widehat h_\ell$ may differ significantly, we set 
$\Xi_{N_\ell}=\eps_\ell=\E(\widehat h_\ell^2)$ for $\ell=0,\dots,L$ and choose the number of samples as $M_0=\ceil{\E(\widehat h_L^2)^{-1}},\,M_\ell=\ceil{\E(\widehat h_L^2)^{-1}\E(\widehat h_\ell^2)\ell^{2(1+\nu)}}$ with $\nu=0.01$ for $\ell=1,\dots,L$. 
The same error equilibration is used for the non-adaptive method, only with $\E(\widehat h_\ell^2)$ replaced by $h_\ell^2$.	
To subtract samples $u_{N_\ell,\eps_\ell,\ell}$ and $u_{N_{\ell-1},\eps_{\ell-1},{\ell-1}}$ in the MLMC estimator and add samples of the same refinement level (but on possibly different stochastic triangulations), we prolong the samples onto a reference grid with refinement size $(\E(\widehat{h}_\ell^2))^{1/2}$.
We consider the test cases with $L=0,\dots,7$ and use $E^L(\widehat u_{N_L,\eps_L,L})$ with $L=9$, as a reference estimate for $\E(u)$.
The RMSE $||E^L(\widehat u_{N_L,\eps_L,L})-\E(u))||_{L^2(\gO;V)}$ is then estimated by averaging $20$ samples of the error $||E^L(\widehat u_{N_L,\eps_L,L})-E^{9}(u_{N_{9},\eps_{9},9})||_V^2$ for $L=0,\dots,7$. To calculate the RMSE, we use a reference grid with $10^3$ equally spaced points in $\D$, thus the error stemming from the interpolation\textbackslash prolongation may be neglected.

\begin{figure}[ht]
	\centering
	\includegraphics[height=0.22\textheight, width=0.49\textwidth,]{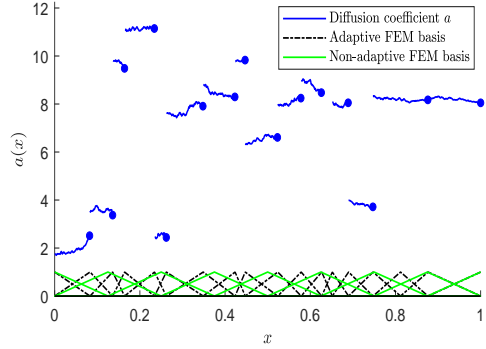}
	\caption{Sample of the diffusion coefficient with Brownian motion covariance and uniformly distributed jumps.}
	\label{fig:BB_Unif}
\end{figure}
As our first numerical example, we use the Brownian motion covariance operator $Q_{BM}$ and i.i.d uniformly distributed jump heights $P_i\sim\cU([0,10])$, 
hence the sampling error $\eps_\ell$ is equal to zero on every level and may be omitted for this scenario.
A sample of the corresponding diffusion coefficient with illustrated adaptive and non-adaptive FEM-basis is given in Fig.~\ref{fig:BB_Unif}.
Fig.~\ref{fig:BB_Unif_conv} indicates that the adaptive algorithm converges considerably faster than the estimator with non-adaptive FEM basis.  
Asymptotically, we see that both adaptive RMSE curves decay with rate nearly one, whereas the non-adaptive methods only show a rate of $\gk\approx 0.75$.
One sees that the adaptive \MLMC\, estimator also has a better time-to-error ratio, so it is possible to reduce the RMSE significantly using a little more computational effort to adjust the FEM basis in each sample.
Surprisingly there is little difference in the convergence speed whether or not we use a bootstrapping algorithm combined with adaptive resp. non-adaptive FEM. 
Here one would expect at least a slightly higher RMSE of the bootstrapping algorithms, but in this example, the error of both bootstrapping estimators is even lower compared to their non-bootstrapping alternatives.  
Naturally, bootstrapping decreases computational time (see Fig.~\ref{fig:BB_Unif_conv}).
\begin{figure}[ht]
	\centering
	\subfigure{\includegraphics[height=0.22\textheight, width=0.49\textwidth,]{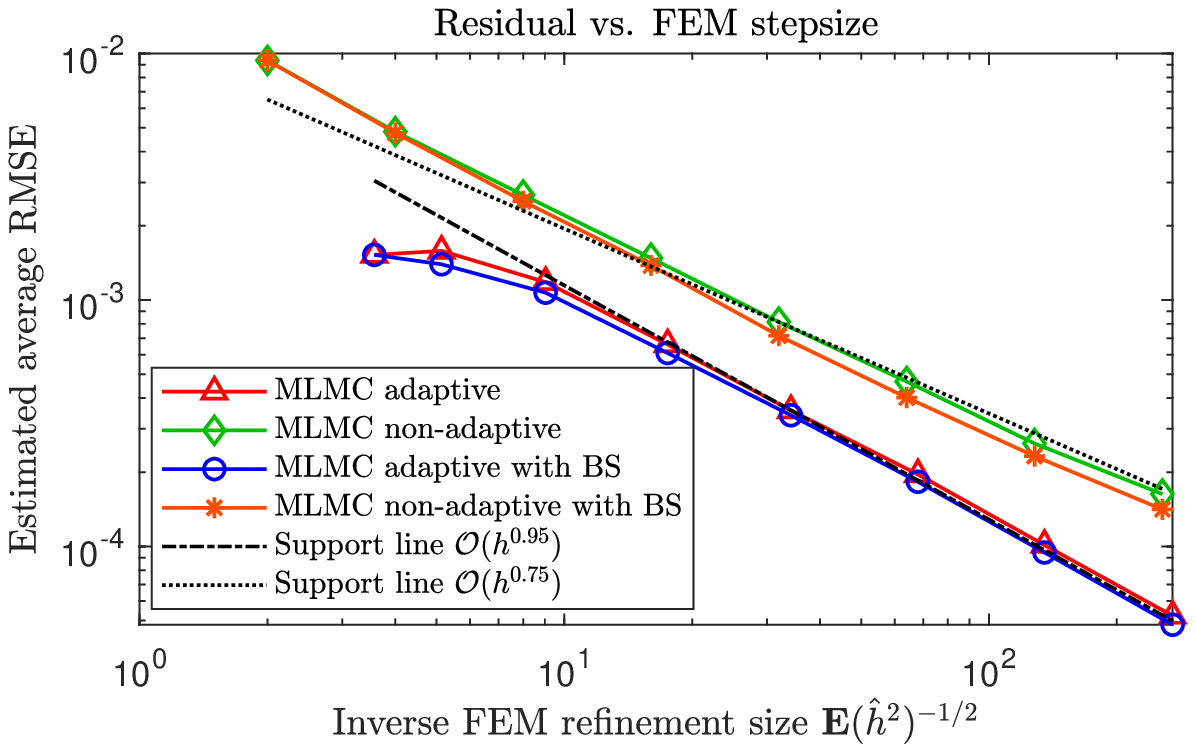}}
	\subfigure{\includegraphics[height=0.22\textheight, width=0.49\textwidth,]{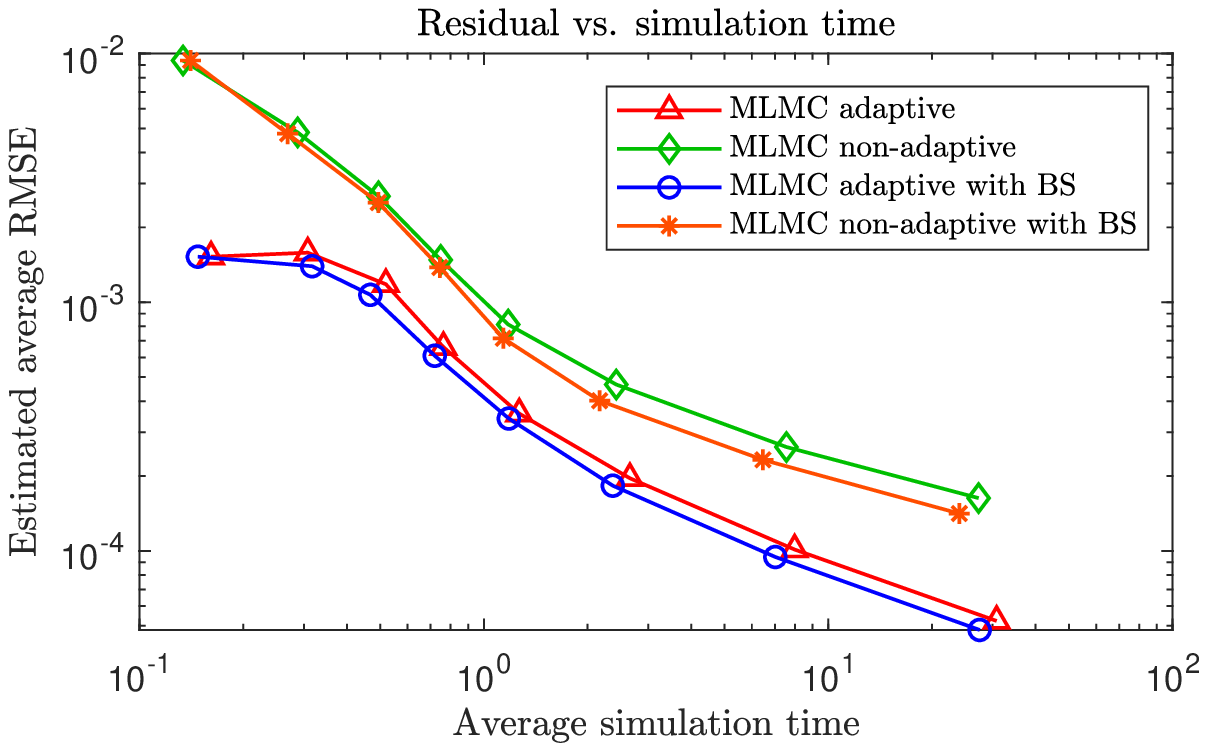}}
	\caption{Left: RMSE with $\cU([0,10])$-distributed jumps, Right: Time-to-error plot.}
	\label{fig:BB_Unif_conv}
\end{figure}

In the next example, we consider the squared exponential covariance operator and a more involved distribution of jump heights, where sampling is rather expensive and may not be realized in a straightforward manner. 
The jump heights $P_i$ now follow a continuous \textit{generalized inverse Gaussian} (GIG) distribution with density
\bee
f_{GIG}(x)=\frac{(\psi/\chi)^{\ol\gl/2}}{2K_{\ol\gl}(\sqrt{\psi\chi})}x^{\ol\gl-1}\exp(-\frac 1 2(\psi x+\chi x^{-1})),\quad x>0
\eee
and parameters $\chi,\psi>0$, $\ol\gl\in\R$, where $K_{\ol\gl}$ is the modified Bessel function of the second kind with $\ol\gl$ degrees of freedom, see \cite{BN78, BNH77}.
As shown in \cite{A82}, sampling this distribution by Acceptance-Rejection is possible, but expensive when $\ol\gl<0$, since the vast majority of outcomes has to be rejected.
We rather generate approximations $\widetilde P_i$ of $P_i$ by a Fourier inversion technique such that $\E(|\widetilde P_i-P_i|^2)\le\eps_\ell$ for a given $\eps_\ell>0$.
For details on the Fourier inversion algorithm, the sampling of GIG distributions and the corresponding error bounds we refer to \cite{BS17}.
The GIG parameters are set as $\psi=0.25, \chi=9$ and $\ol\gl=-1$, the resulting density $f_{GIG}$ and a sample of the diffusion coefficient are given in Fig.~\ref{fig:BB_GIG}. 
\begin{figure}[ht]
	\centering
	\subfigure{\includegraphics[height=0.22\textheight, width=0.49\textwidth,]{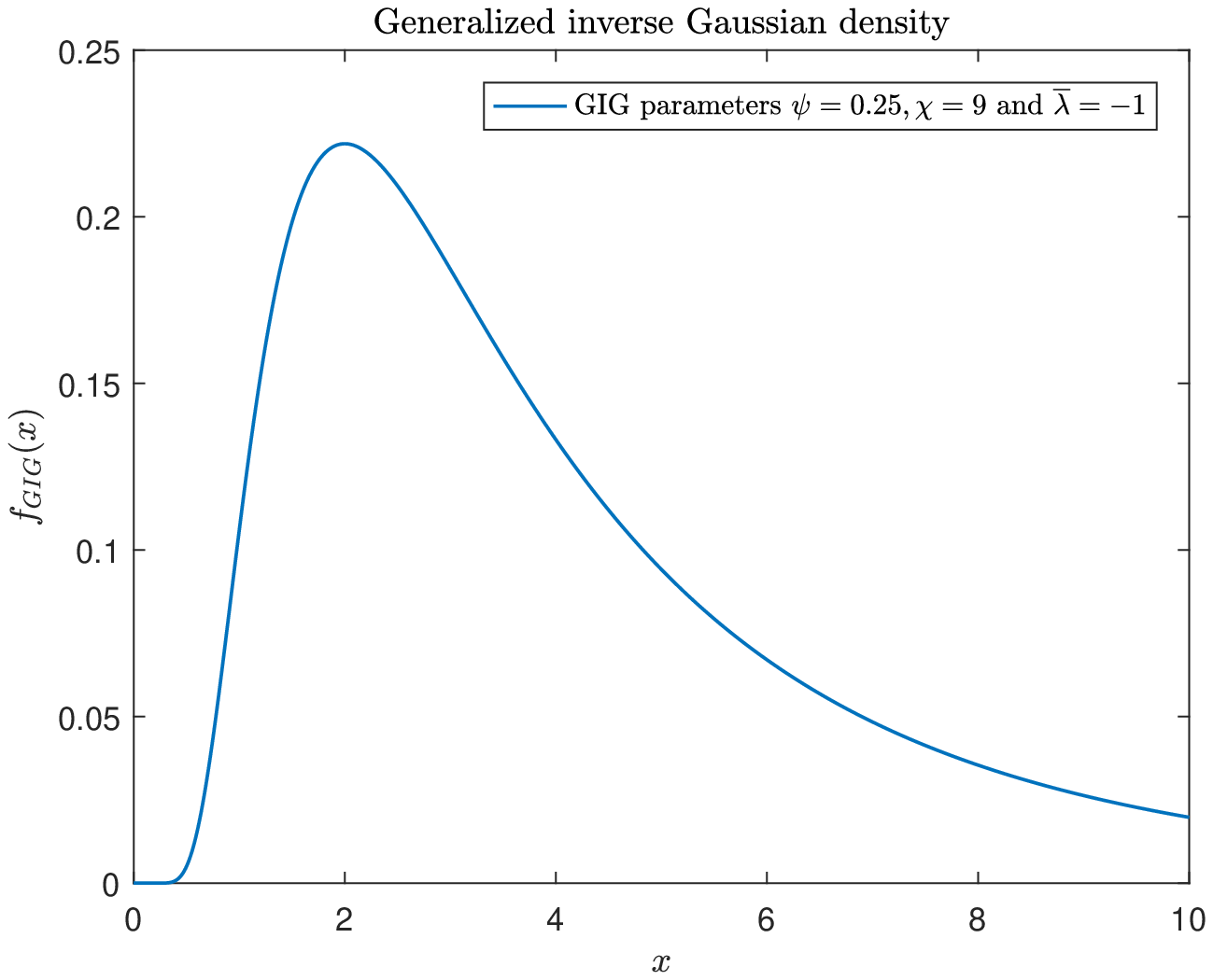}}
	\subfigure{\includegraphics[height=0.22\textheight, width=0.49\textwidth,]{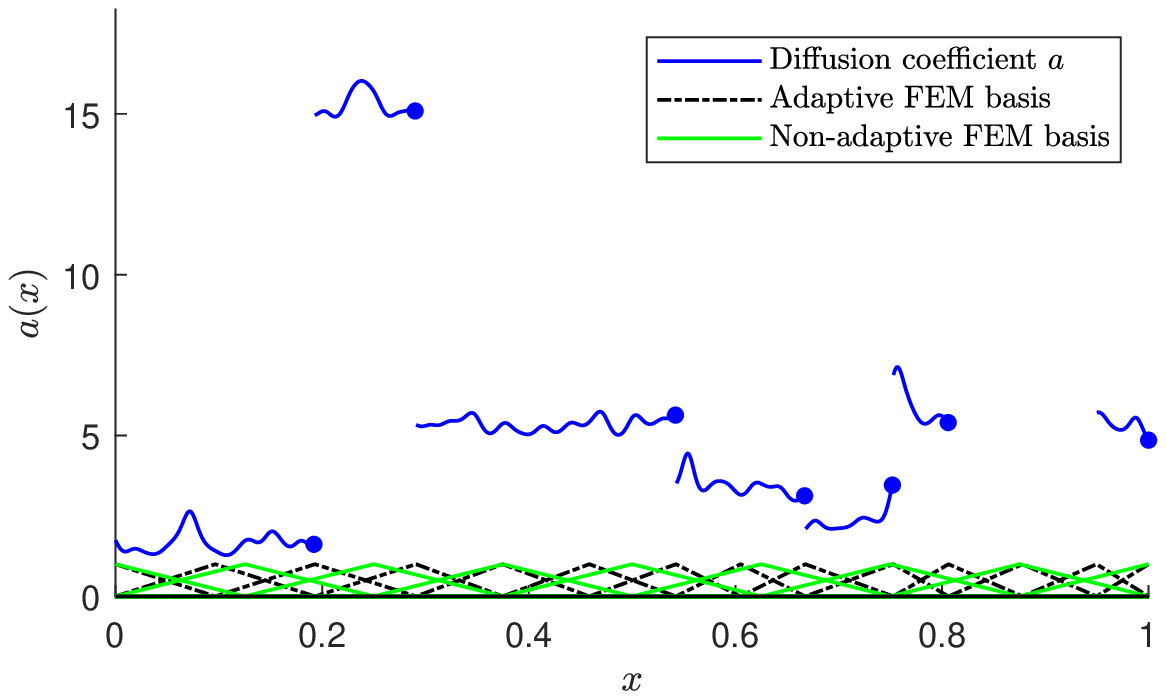}}
	\caption{Left: GIG density, right: sample of the diffusion coefficient with log-squared exponential covariance and GIG distributed jumps.}
	\label{fig:BB_GIG}
\end{figure}
The error curves show a similar behavior compared to the first example, with asymptotic error rates of $1$  resp. $0.75$ for the adaptive resp. non-adaptive algorithms, see Fig.~\ref{fig:BB_GIG_conv}. 
Again, bootstrapping tends to produce a slightly lower RMSE.
The expensive sampling from the GIG distribution causes increased computational times, 
which entails that the bootstrapping is even more favorable in a setting with a rather challenging jump height distributions. 
\begin{figure}[ht]
	\centering
	\subfigure{\includegraphics[height=0.22\textheight, width=0.49\textwidth,]{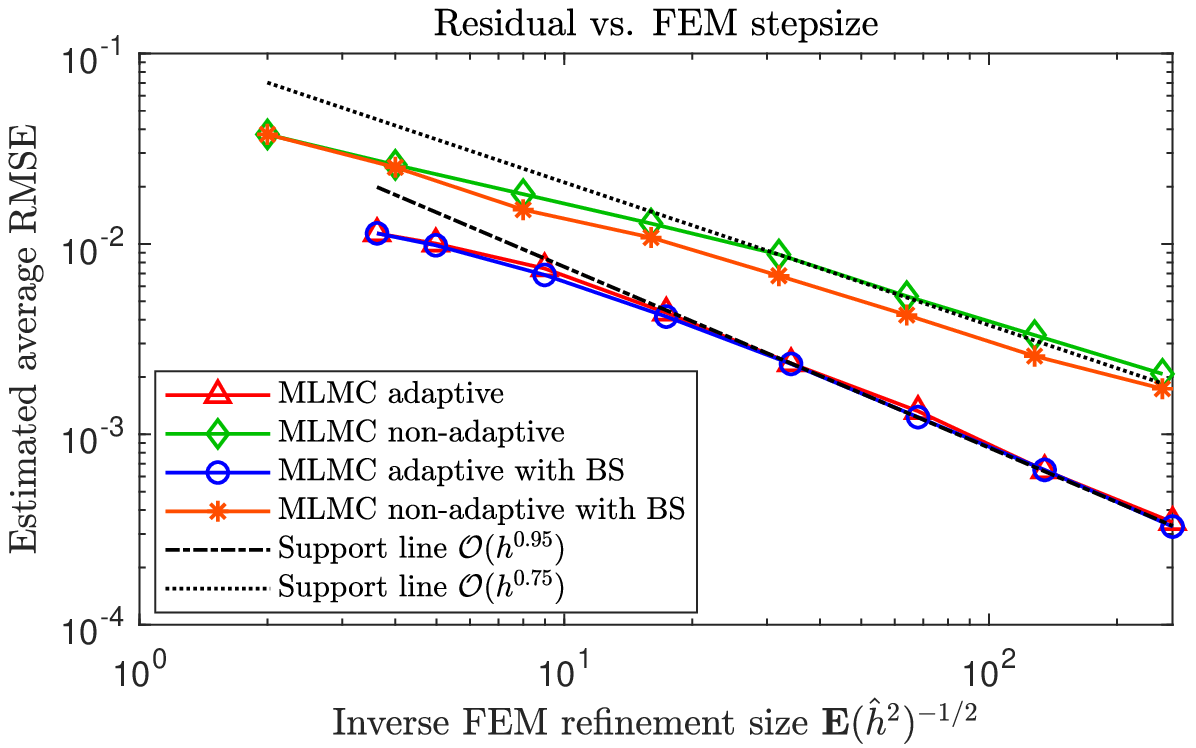}}
	\subfigure{\includegraphics[height=0.22\textheight, width=0.49\textwidth,]{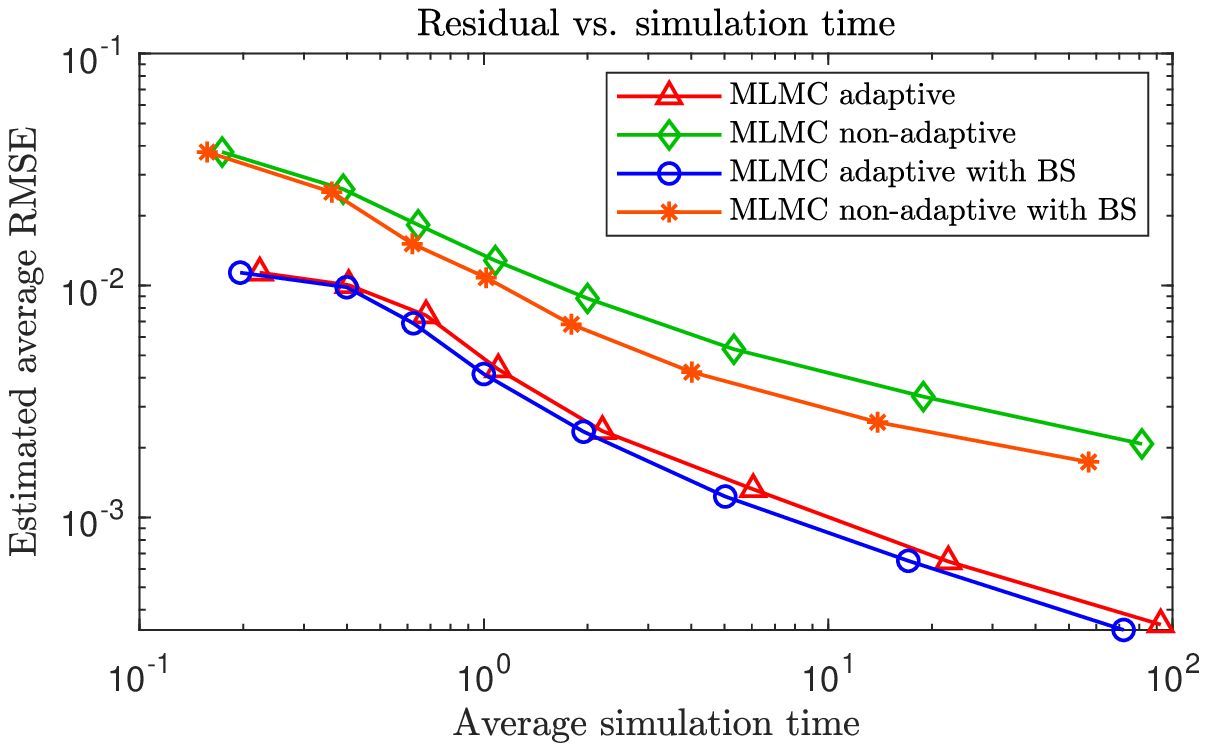}}
	\caption{Left: RMSE of the example with GIG-distributed jumps, Right: Time-to-error plot.}
	\label{fig:BB_GIG_conv}
\end{figure}

The Gaussian random field $W:\gO\times\D\to\R$ with the Brownian motion covariance operator is almost surely not differentiable in a.e.~$x\in\D$, but only H\"older continuous. 
In addition, the covariance of the random variables $W(\cdot,x_1)$ and $W(\cdot,x_2)$, where $x_1,x_2\in\D$ is given by the kernel function $\min(x_1,x_2)$.
For a fixed distance between $x_1$ and $x_2$, this implies that the correlation in the random field increases as one moves to the right boundary of the domain.
In some applications, however, one might instead want a random field with stationary correlation structure and/or more spatial regularity.    
This can be achieved with the introduced jump-diffusion coefficient by using, for instance, $Q_{SE}$ or another Mat\'ern class covariance operator.
These covariance operators generate random stationary correlated random fields and also increase the regularity of $W$ in $\D$.  
It is further possible to vary the position and magnitude of the discontinuities of $a$ to model certain desirable characteristics of the diffusion.   
For example, one could enforce only one jump per sample which is concentrated in some small neighborhood located around a single point in $\D$.
The corresponding jump heights on each partition may then also be chosen concentrated around certain values to model, for instance, transitions in heterogeneous or fractured media. 

\subsection{Numerical results in 2D}
In the two-dimensional setting, we work on the domain $\D=(0,1)^2$, with homogeneous Dirichlet or mixed Neumann-Dirichlet boundary conditions and we assume that the deterministic part of the drift coefficient is zero ($\bar a\equiv0$). 
The Gaussian part of $a$ is given by the \KL expansion
with spectral basis given by
\bee
\eta_i:=v\exp(-\pi^2i^2r^2),\quad e_i(x):=\sin(\pi ix_1)\sin(\pi ix_2),
\eee
with correlation length $r>0$ and total variance  $v>0$.
This basis is related to the two-dimensional heat kernel 
\bee
G(t,x,y):[0,\infty)\times\D^2\to\R^+,\quad (t,x,y)\mapsto\frac{1}{4\pi t}\exp(-\frac{-||x-y||_2^2}{4t})
\eee
in the sense that it solves the associated integral equation for $t=r^2/2$:
\bee
v\int_{\D}\exp(-\frac{||x-y||_2^2}{2r^2})e_i(y)dy=\eta_i e_i(x),\quad i\in\N
\eee
with the boundary condition $e_i=0$ on $\partial\D$, see \cite{GN13}. 
Compared with a Gaussian field generated by a squared exponential covariance operator, this field shows a very similar behavior, except that it is zero on the boundary. It, further, has the advantage, that all expressions are available in closed form and we forgo the numerical approximation of the eigenbasis. For all experiments in this section we use the parameters $v=0.25$ and $r=0.02$.
As before, we consider a log-Gaussian random field, meaning $\Phi(w)=\exp(w)$.
To illustrate the flexibility of a jump-diffusion coefficient $a$ as in Def.~\ref{def:a}, we vary the random partitioning of $\D$ for each example and give a detailed description below.
Again, we approximate the entries of the stiffness matrix by the midpoint rule on each simplex $K$. 
To ensure that this is sufficient, we also tested a four-point Gauss-Legendre quadrature rule for triangular domains, which did not change the outcomes significantly.
As in the previous subsection, we prolong linearly on a fixed grid to add and subtract the generated samples of the \MLMC\ estimator. The reference grid for the RMSE consists of $400\times 400$ equally spaced points in the domain $\D=(0,1)^2$, which yields a negligible interpolation error.
The multilevel approximation parameters are identical to the 1D example, except that $\ol h_\ell=h_\ell=\frac{2}{5}2^{-\ell}$ and we now calculate the reference solution on level $7$ to estimate the RMSE (by averaging $10$ independent \MLMC\, estimations) up to level $L=4$ or $L=5$, depending on the example. All RMSE curves are plotted against the inverse estimated refinement $\E(\widehat h^2)^{-1/2}$ on the abscissa. 

In the first 2D example, the random partitions $\cT$ of $\D$ are generated by random lines through the domain. 
More precisely, we sample independent Poisson random variables $\cP_x, \cP_y\sim Poi(1)$ and a total of $2(\cP_x+\cP_y+2)$ independent $\cU([0,1])$-distributed random variables
$U_1,\dots,U_{2(\cP_x+\cP_y+2)}$.
The first $\cP_x+1$ uniform random variables represent the jump positions on $(0,1)\times\{0\}$, 
the second set $U_{\cP_x+2},\dots,U_{2\cP_x+2}$ are the positions of the discontinuities on the opposing axis $(0,1)\times\{1\}$ in $\partial\D$. 
We connect opposing points in ascending order by straight lines to obtain a vertical random partition of $\D$. 
Analogously, the horizontal splitting is achieved by distributing and connecting the remaining $2\cP_y+2$ uniform random variables on the sets $\{0\}\times(0,1)$ and $\{1\}\times(0,1)$.
As we obtain an average of $\E((\cP_x+2)(\cP_y+2))=9$ partition elements of uniformly distributed size and location, $\gl$ may be set as $\gl=9\gL$ in this example, 
where $\gL$ is the Lebesgue measure on $(\D,\cB(\D))$.

To each of the $(\cP_x+2)(\cP_y+2)$ random quadrangles we assign a jump height $P_i$, where the sequence $(P_i,i\in\N)$ is i.i.d. $\cU([0,5])$-distributed.
This specific structure of $a$ may be used, for example, to model stationary flows through heterogeneous media, where the hydraulic conductivity varies on sub-domains of the medium.
\begin{figure}[ht]
	\centering
	\subfigure{\includegraphics[height=0.22\textheight, width=0.49\textwidth,]{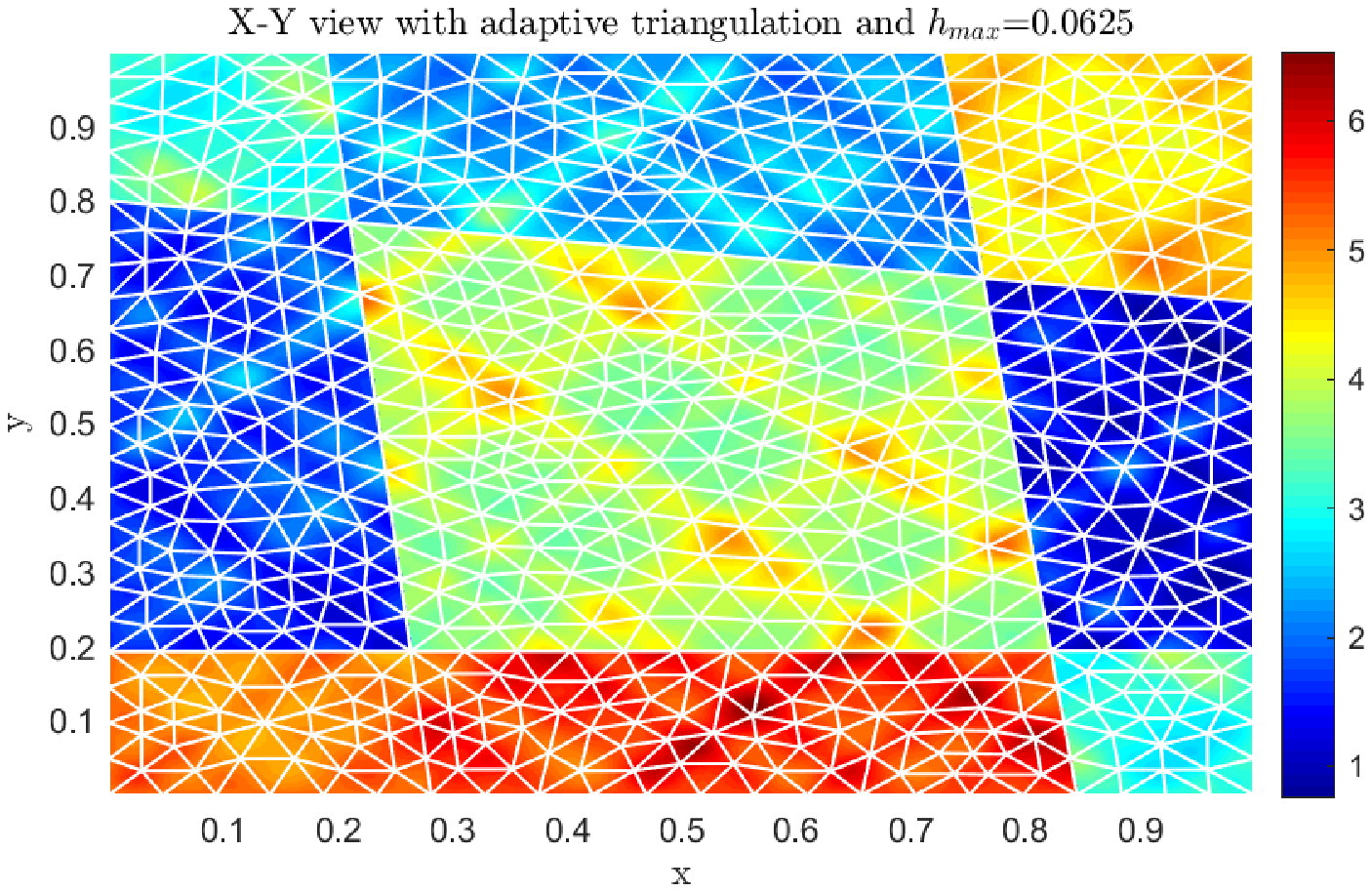}}
	\subfigure{\includegraphics[height=0.22\textheight, width=0.49\textwidth,]{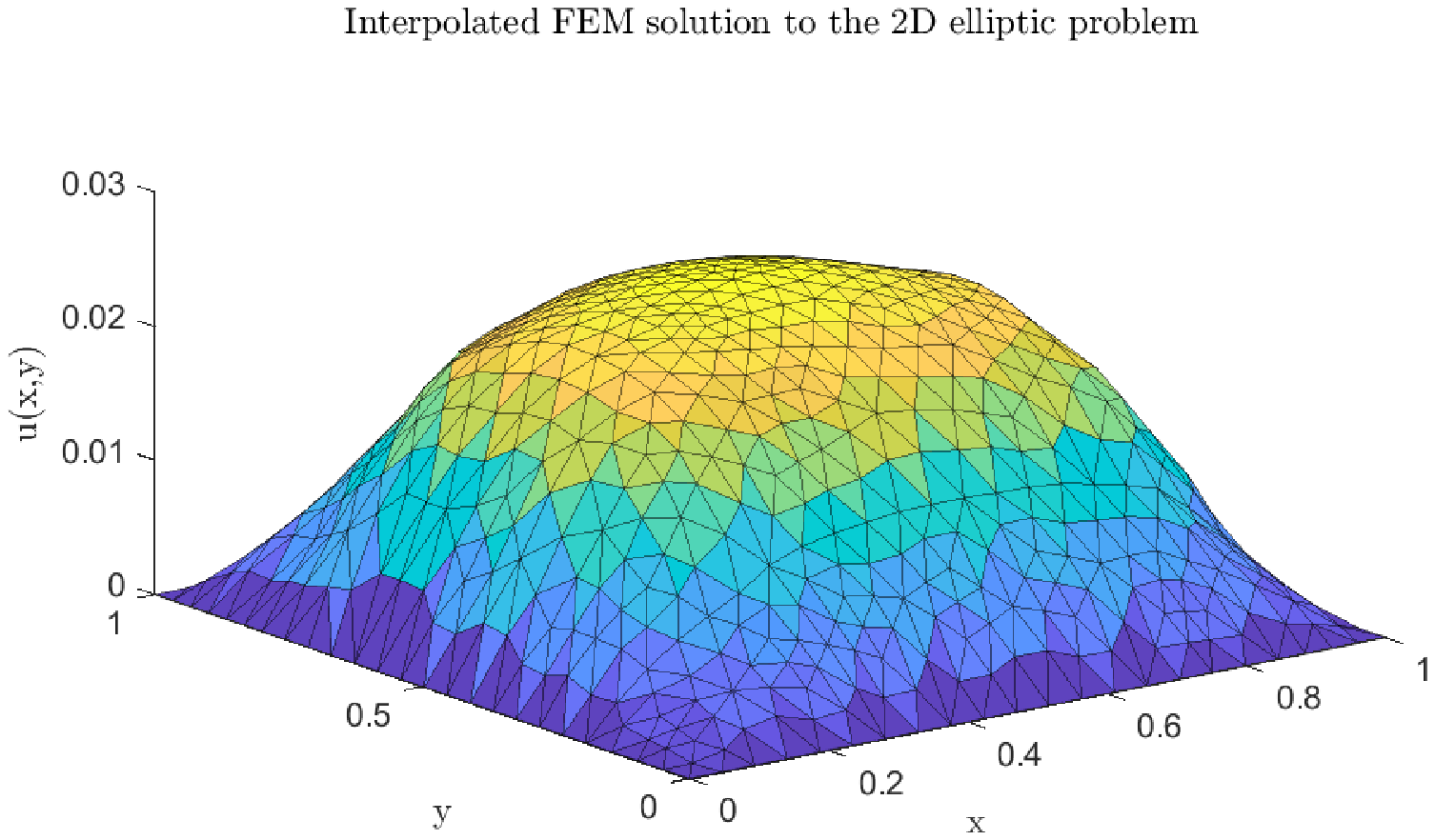}}
	\caption{Left: Sample of the 2D diffusion coefficient in heterogeneous media. Right: Approximated FEM solution to the sample of $a$.}
	\label{fig:2D_cm}
\end{figure}
We assume homogeneous Dirichlet-boundary conditions on $\gG_1=\partial\D$ and $f\equiv1$ as source term. 
A sample of the diffusion coefficient and the FEM approximation is given in Fig.~\ref{fig:2D_cm}. 
Compared to a solution with constant coefficient, the influence of the discontinuous diffusion is clearly visible in the contour of the approximated solution.

Fig.~\ref{fig:2D_conv0} shows that the adaptive \MLMC\, and non-adaptive \MLMC\, algorithm perform quite similar, where the asymptotic error rate of the nonadaptive methods is slightly lower with $0.85$. Compared to that, we recover a convergence rate of $0.9$ and lower absolute errors for the adaptive method.
In both cases, bootstrapping and normal \MLMC\, sampling produces almost identical errors which results in the best time-to-error ratio of the adaptive bootstrap \MLMC\, estimator, see Fig.~\ref{fig:2D_conv0}.
\begin{figure}[ht]
	\centering
	\subfigure{\includegraphics[height=0.22\textheight, width=0.49\textwidth,]{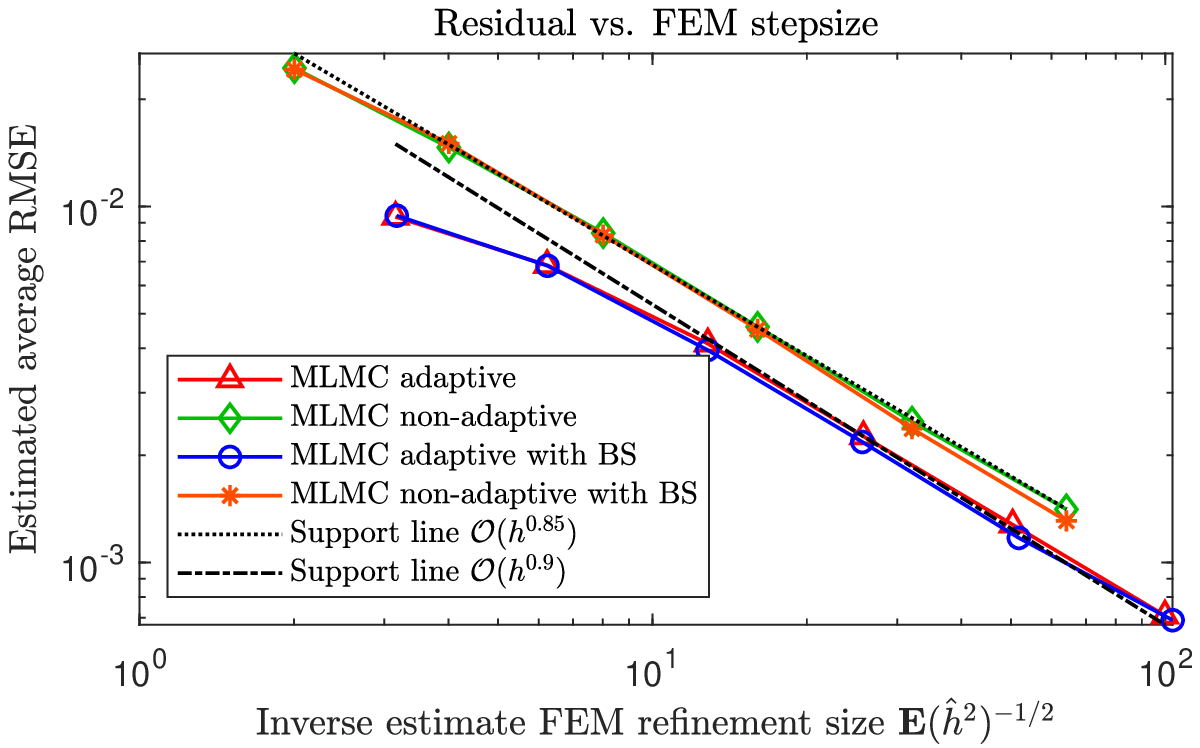}}
	\subfigure{\includegraphics[height=0.22\textheight, width=0.49\textwidth,]{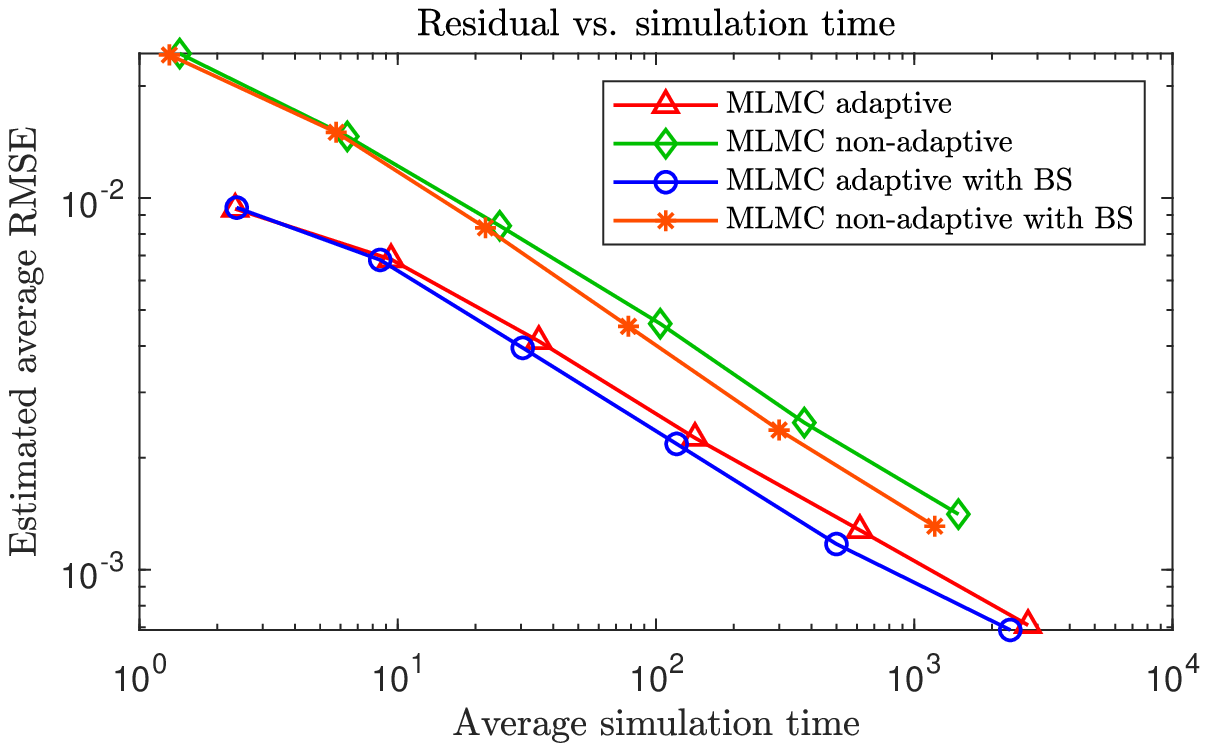}}
	\caption{Left: RMSE of the heterogeneous media example, Right: Time-to-error plot.}
	\label{fig:2D_conv0}
\end{figure}

In the second 2D example, we aim to mimic the structure of a fractured porous medium.
To this end, we set $\gl=5\gL$ and sample accordingly $\tau=5$ uniformly distributed random points $x_1,\dots,x_5$ on the domain $\D$.
Then, for each point $x_i$, a random length $l_i$ with distribution $\cU([0.5,1.5])$ is generated.
We sample, further, five random angles $\ga_1,\dots,\ga_5$ with uniform distribution on the set $[-\frac{\pi}{9},-\frac{\pi}{36}]\cup[\frac{\pi}{36},\frac{\pi}{9}]$.
We define $x_i$ as the center of a line with length $l_i$ rotated by $\ga_i$, where three of the five lines are orientated horizontally and the remaining two lines are vertical. 
Finally, the line segments outside of $\D$ are removed and each random line is assigned a positive preset width of $0.04$, which results in a "trench structure" of the diffusion coefficient as depicted in Fig.~\ref{fig:2D_fpm}.
On the trenches, we set the jump height to $P_1=100$, while the jump heights on the remaining quadrangles of the partition is set to $1$. 
Fixing the jump heights captures increasing permeability in the cracks of a certain medium, the Gaussian field still accounts for some uncertainty within each partition element of the domain. 
The source function is given as $f\equiv5$. 
We  split $\partial\D$ by $\gG_1:=\{0,1\}\times[0,1]$ and $\gG_2:=(0,1)\times\{0,1\}$ and impose the (pathwise) mixed Dirichlet-Neumann boundary conditions  
\be\label{eq:bound}
u(\go,\cdot)=\begin{cases}
	0.1 \text{\quad on $\{0\}\times[0,1]$}\\
	0.3 \text{\quad on $\{1\}\times[0,1]$}\\
\end{cases}\quad\text{and}\quad a(\go,\cdot)\vv n \cdot\grad u(\go,\cdot)=0 \text{\quad on $\gG_2$}
\ee
for each $\go\in\gO$.
A sample of the approximated solution is displayed in Fig.~\ref{fig:2D_fpm}.
\begin{figure}[ht]
	\centering
	\subfigure{\includegraphics[height=0.22\textheight, width=0.49\textwidth,]{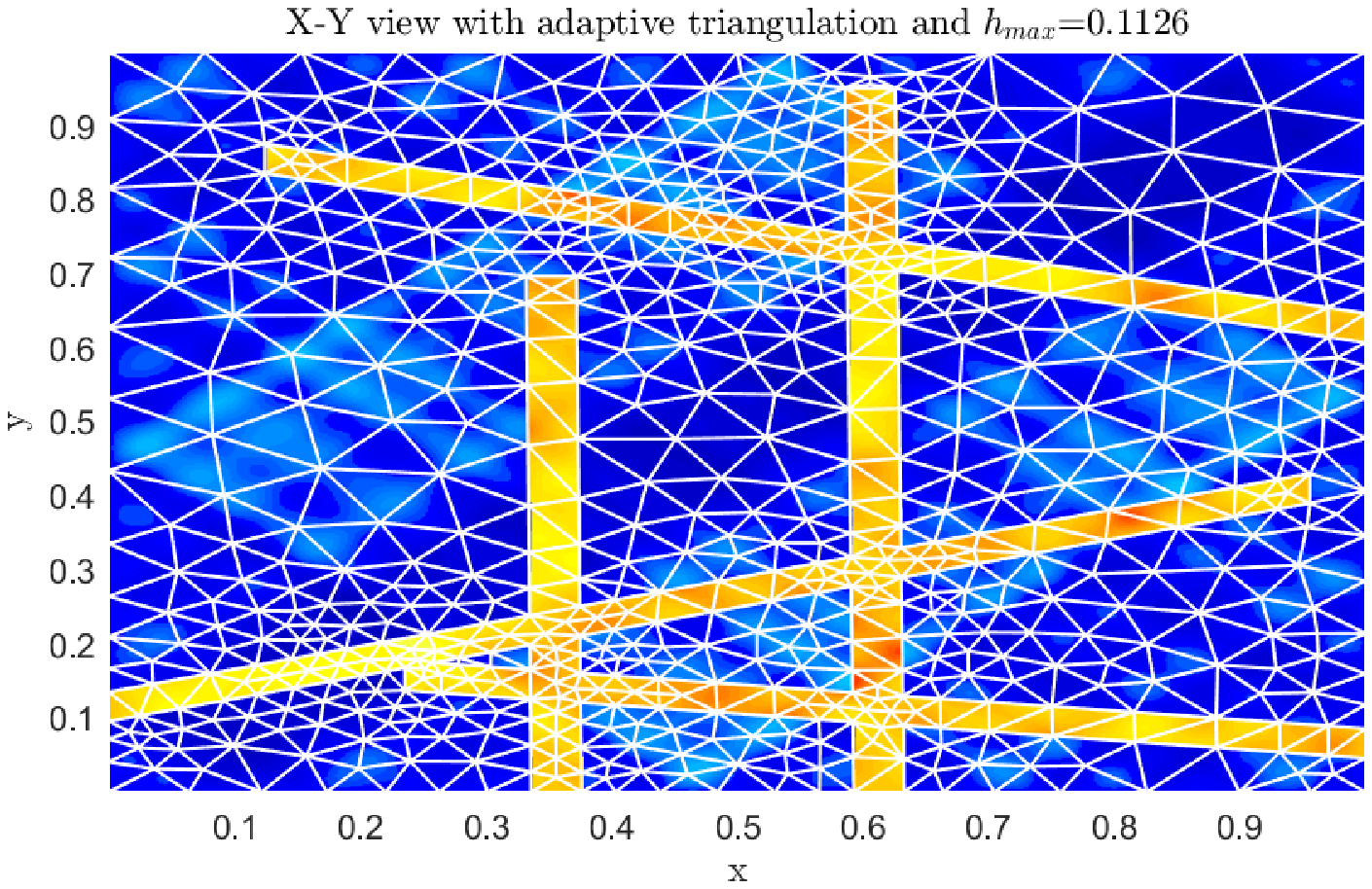}}
	\subfigure{\includegraphics[height=0.22\textheight, width=0.49\textwidth,]{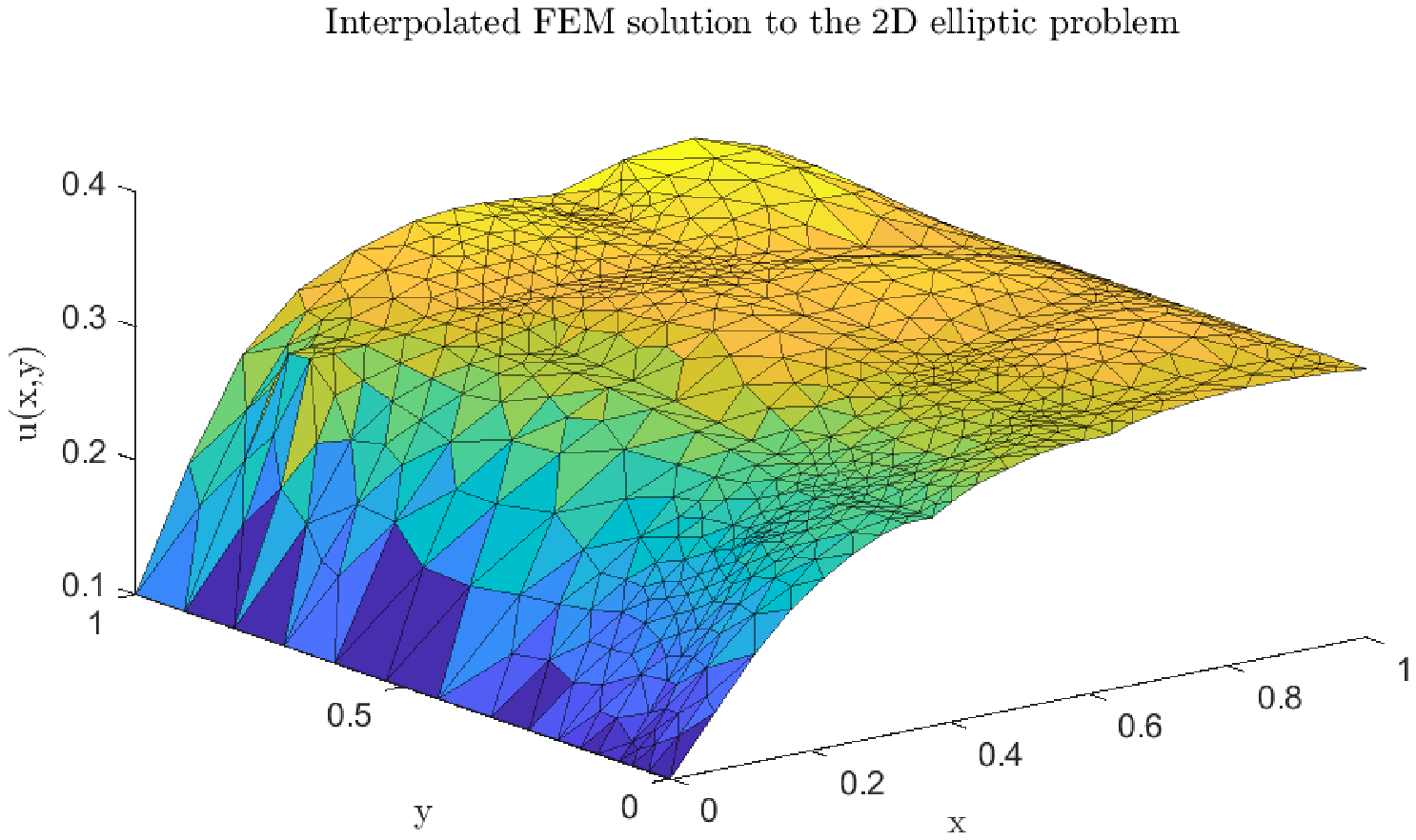}}
	\caption{Left: Sample of the fractured porous medium diffusion coefficient with adaptive triangulation.
		Right: Approximated FEM solution to the sample of $a$.}
	\label{fig:2D_fpm}
\end{figure}

Compared to the first 2D example, there is now a larger gap between the RMSE curves of the adaptive and non-adaptive estimators, see Fig.~\ref{fig:conv_fpm}.
The asymptotic rate for the error is again close to order one for both adaptive methods, while we obtain $0.6$ for the non-adaptive algorithms. This is possibly due to the higher magnitude of the discontinuities in the diffusion coefficient compared to the first example.
Bootstrapping now leads to a higher RMSE in each algorithm and this effect is more pronounced in the adaptive setting. Asymptotically, the rates of both bootstrapping estimators remain comparable to standard \MLMC.		
An adaptive triangulation for samples of this particular diffusion coefficient often entails very fine meshes, even if the desired maximum diameter of each triangle is comparably high, as Fig.~\ref{fig:2D_fpm} and Fig.~\ref{fig:conv_fpm} illustrate.
Due to this increase in complexity when using adaptive FEM, the simulation times for the respective estimators are now considerably longer as in the previous scenario, see Fig.~\ref{fig:conv_fpm}.
Nevertheless, the adaptive methods still have significantly better time-to-error ratios, but now the standard adaptive method outperforms the corresponding bootstrapping estimator.
\begin{figure}[ht]
	\centering
	\subfigure{\includegraphics[height=0.22\textheight, width=0.49\textwidth,]{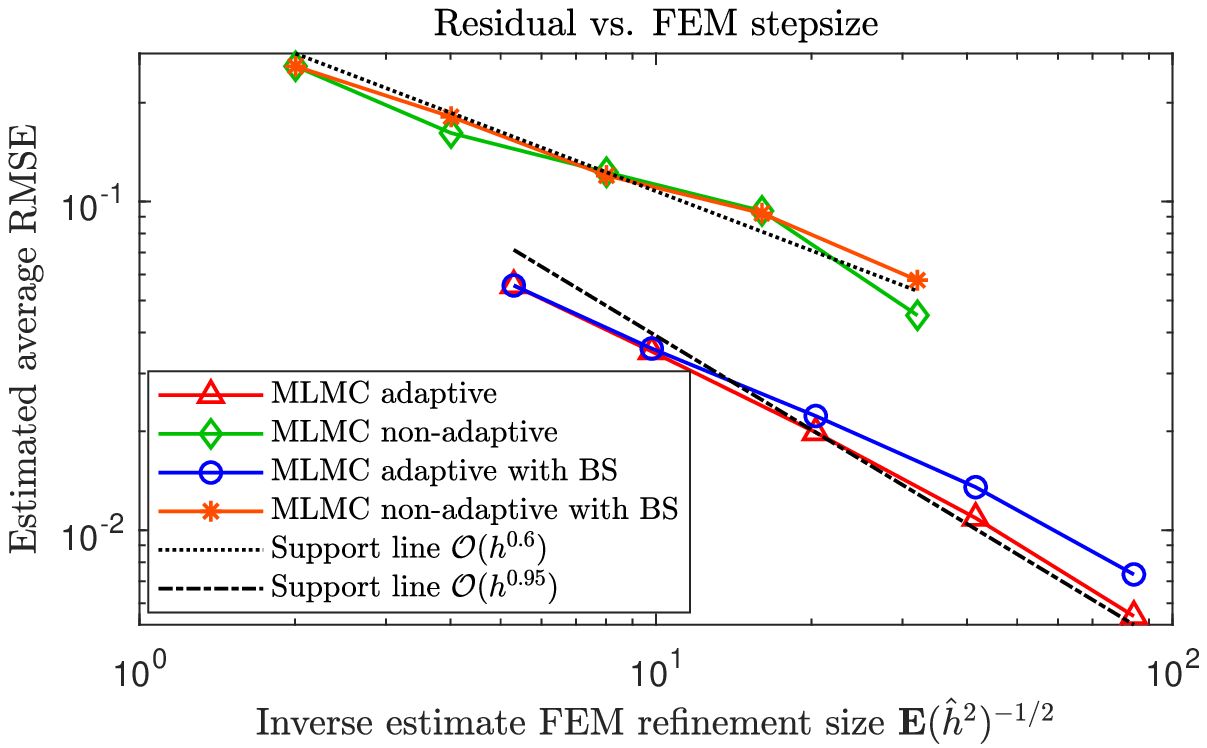}}
	\subfigure{\includegraphics[height=0.22\textheight, width=0.49\textwidth,]{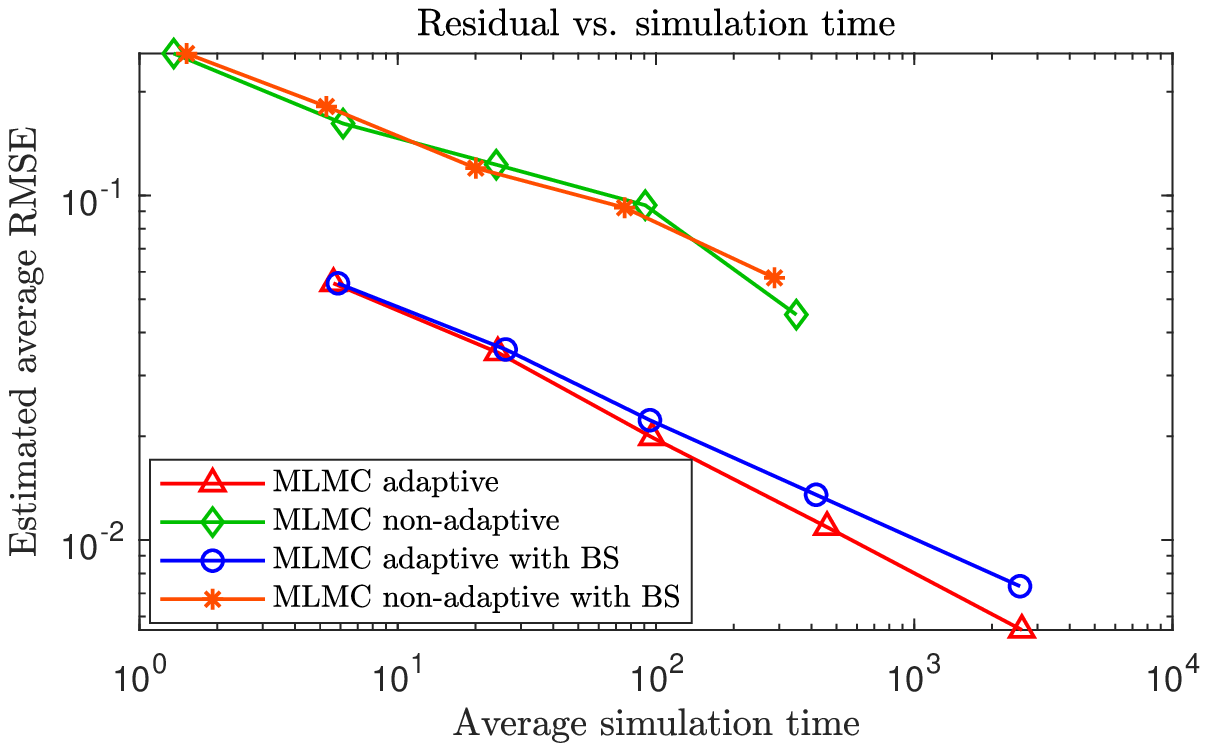}}
	\caption{Left: RMSE of the fractured porous media example, Right: Time-to-error plot.}
	\label{fig:conv_fpm}
\end{figure}

As a last 2D example, we discuss a medium with inclusions.
To this end, we sample a discrete uniformly-distributed random variable $\tau$ where $\tau\in\{1,2,3,4\}$ and define $\gl:=\frac{\E(\tau)}{0.64}\gL|_{(0.1,0.9)^2}$.
Scaling and restricting the Lebesgue measure on $(0.1,0.9)^2$ means that we now draw $\tau$-many uniformly distributed points within the the sub-domain $(0.1,0.9)^2\subset\D$.
To each of this points we assign a circle of random radius. The radii are $\cU([0.075,0.1])$-distributed. 
On each circle we assign a jump height of 1, while this parameter is set to $20$ outside of the circles.   
We assume the same Neumann--Dirichlet boundary conditions as in our second 2D-example (see Eq.~\eqref{eq:bound}) and as a source term we set $f\equiv10$. 
A sample of this jump-diffusion coefficient with corresponding FEM solution is shown in Fig.~\ref{fig:2D_circle}.
\begin{figure}[ht]
	\centering
	\subfigure{\includegraphics[height=0.22\textheight, width=0.49\textwidth,]{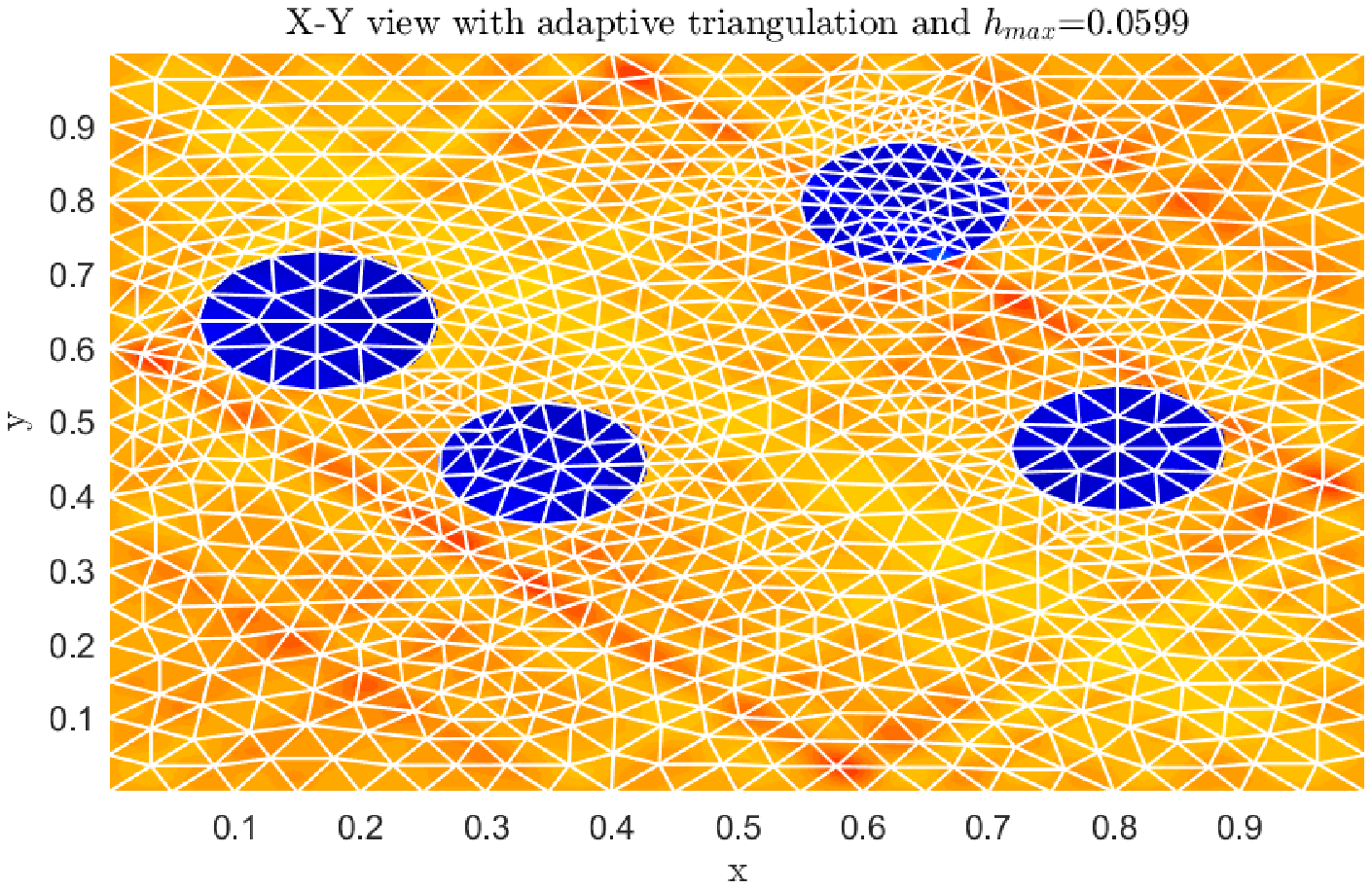}}
	\subfigure{\includegraphics[height=0.22\textheight, width=0.49\textwidth,]{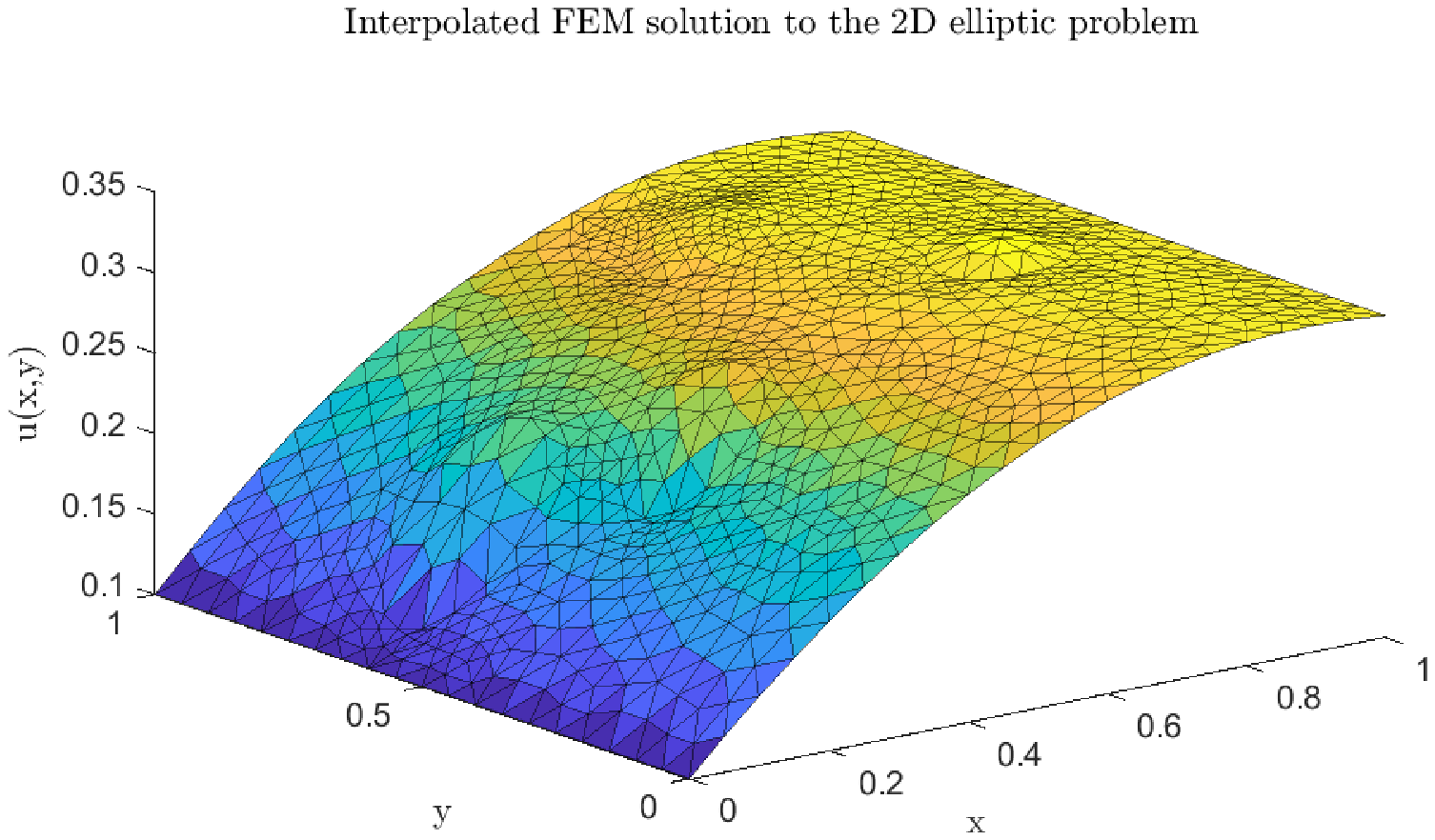}}
	\caption{Left: Sample of the diffusion coefficient of a medium with inclusions with adaptive triangulation.
		Right: Approximated FEM solution to the sample of $a$.}
	\label{fig:2D_circle}
\end{figure}

As Fig.~\ref{fig:conv_circle} indicates, we obtain a similar behavior of convergence as in the previous example: 
The RMSE of the adaptive estimators is again significantly lower on all levels and the non-adaptive has again an asymptotic RMSE of order $0.6$. 
For the adaptive methods, we obtain log-linear error decay of order $\cO(h|\log(h)|^{\frac{1}{2}})$.
This corresponds to the expected pathwise rate for an adaptive FEM solution of this diffusion problem, as the discontinuities have $C^2$-boundaries.  
Bootstrapping slightly reduces the RMSE of the non-adaptive method, but has little effect on the adaptive \MLMC\, estimator. 
The computational complexity of this scenario is comparable to the heterogeneous media example  and thus significantly lower as in case of the fractured porous medium. 
Finally, the adaptive algorithms again attain better time-to-error with the bootstrapping estimator slightly outperforming the standard adaptive method, see Fig.~\ref{fig:conv_circle}.
\begin{figure}[ht]
	\centering
	\subfigure{\includegraphics[height=0.22\textheight, width=0.49\textwidth,]{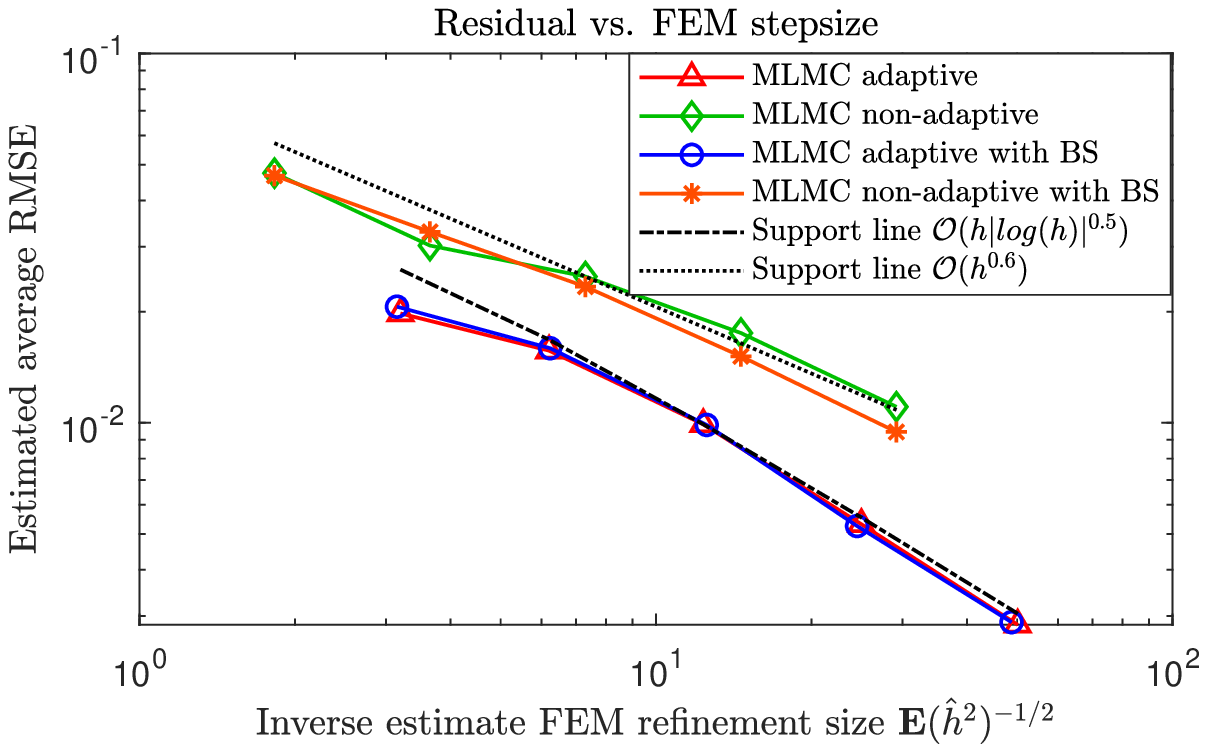}}
	\subfigure{\includegraphics[height=0.22\textheight, width=0.49\textwidth,]{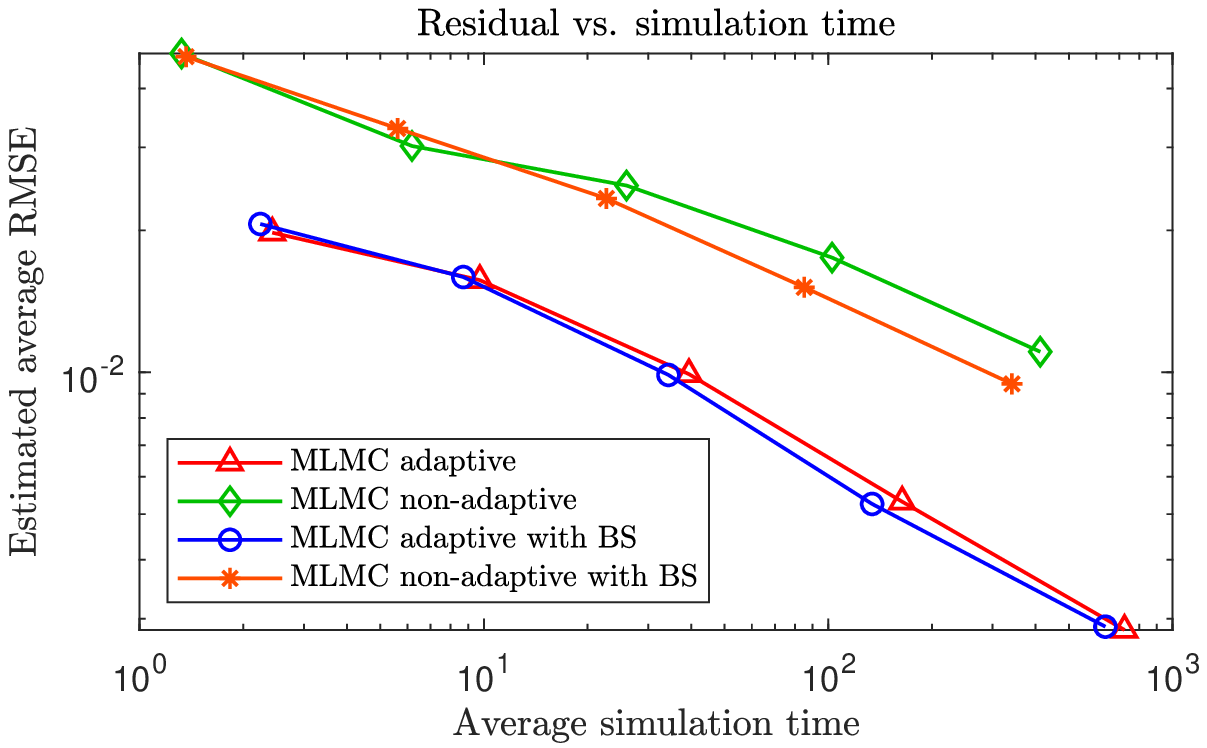}}
	\caption{Left: RMSE of the example for a medium with inclusions, Right: Time-to-error plot.}
	\label{fig:conv_circle}
\end{figure}

\section*{Acknowledgments}
The research leading to these results has received funding from the German Research Foundation (DFG) as part of the Cluster of Excellence in Simulation Technology (EXC 310/2) at the University of Stuttgart and by the Juniorprofessorship program of Baden--W\"urttemberg, and it is gratefully acknowledged. The authors would like to thank two anonymous referees for their remarks which led to a significant improvement of the manuscript.  

\bibliographystyle{siamplain}
\bibliography{elliptic_jump_diffusion}

\end{document}